\documentclass[preprint,11pt,letterpaper]{elsarticle}

\biboptions{authoryear}

\usepackage{multirow}
\usepackage{booktabs}
\usepackage[colorlinks=true,linkcolor=black,citecolor=black]{hyperref}

\usepackage{amsmath}
\usepackage{amssymb}
\usepackage{amsthm}
\usepackage{thmtools}

\usepackage{fullpage}
\usepackage[onehalfspacing]{setspace}

\usepackage[labelfont=bf]{caption}
\usepackage{subcaption}
\usepackage{soul}
\usepackage{tikz}
\usepackage{pgfplots}\pgfplotsset{compat=1.14} % set pgfplots version

\DeclareMathOperator*{\argmin}{arg\,min}
\DeclareMathOperator*{\argmax}{arg\,max}
\DeclareMathOperator*{\minimize}{minimize}
\DeclareMathOperator*{\maximize}{maximize}
\DeclareMathOperator*{\sgn}{sgn}

\newtheorem{lem}{Lemma}
\newtheorem{thm}{Theorem}
\newtheorem{pro}{Proposition}

\theoremstyle{definition}
\newtheorem{rmk}{Remark}
\newtheorem{assumption}{Assumption}
\newtheorem{exmp}{Example}

\renewenvironment{proof}[1][\proofname]{{\bfseries #1. }}{\qed} % makes "proof" bold and allows "proof of..."

\newcommand{\ignore}[1]{}
\newcommand{\bA}{ \mathbf{A} }
\newcommand{\bAh}{ \mathbf{\hat{A}} }

\newcommand{\ba}{ \mathbf{a} }
\newcommand{\bab}{ \mathbf{\bar{a}} }
\newcommand{\bah}{ \mathbf{\hat{a}} }

\newcommand{\bb}{ \mathbf{b} }

\newcommand{\bc}{ \mathbf{c} }

\newcommand{\beee}{ \mathbf{e} }
\newcommand{\bff}{ \mathbf{f} }
\newcommand{\bg}{ \mathbf{g} }

\newcommand{\bs}{ \mathbf{s} }
\newcommand{\bt}{ \mathbf{t} }
\newcommand{\bu}{ \mathbf{u} }
\newcommand{\bv}{ \mathbf{v} }
\newcommand{\bw}{ \mathbf{w} }
\newcommand{\bx}{ \mathbf{x} }
\newcommand{\bxh}{ \mathbf{\hat{x}} }

\newcommand{\by}{ \mathbf{y} }
\newcommand{\bz}{ \mathbf{z} }
\newcommand{\bzero}{ \mathbf{0} }
\newcommand{\balpha}{ \boldsymbol{\alpha} }

\newcommand{\balphah}{ \boldsymbol{\hat{\alpha}} }

\newcommand{\bGamma}{ \boldsymbol{\Gamma} }
\newcommand{\bGammah}{ \boldsymbol{\hat{\Gamma}} }
\newcommand{\bpi}{ \boldsymbol{\pi} }
\newcommand{\bmu}{ \boldsymbol{\mu} }
\newcommand{\blambda}{ \boldsymbol{\lambda} }
\newcommand{\bvarphi}{ \boldsymbol{\varphi} }

\newcommand{\bxi}{ \boldsymbol{\xi} }
\newcommand{\bOmega}{ \boldsymbol{\Omega} }
\newcommand{\bTheta}{ \boldsymbol{\Theta} }

\journal{European Journal of Operational Research}

\begin{document}

\begin{frontmatter}

\title{Inverse optimization for the recovery of constraint parameters}

\author[uoftaddress]{Timothy C.Y. Chan}
\ead{tcychan@mie.utoronto.ca}

\author[uoftaddress,smhaddress]{Neal Kaw\corref{mycorrespondingauthor}}
\cortext[mycorrespondingauthor]{Corresponding author}
\ead{nkaw@mie.utoronto.ca}

\address[uoftaddress]{Department of Mechanical and Industrial Engineering, University of Toronto, Toronto, Ontario M5S 3G8, Canada}
\address[smhaddress]{Li Ka Shing Centre for Healthcare Analytics Research \& Training, St. Michael's Hospital, Toronto, Ontario M5C 3G7, Canada}

\begin{abstract}
Most inverse optimization models impute unspecified parameters of an objective function to make an observed solution optimal for a given optimization problem with a fixed feasible set. We propose two approaches to impute unspecified left-hand-side constraint coefficients in addition to a cost vector for a given linear optimization problem. The first approach identifies parameters minimizing the duality gap, while the second minimally perturbs prior estimates of the unspecified parameters to satisfy strong duality, if it is possible to satisfy the optimality conditions exactly. We apply these two approaches to the general linear optimization problem. We also use them to impute unspecified parameters of the uncertainty set for robust linear optimization problems under interval and cardinality constrained uncertainty. Each inverse optimization model we propose is nonconvex, but we show that a globally optimal solution can be obtained either in closed form or by solving a linear number of linear or convex optimization problems.
\end{abstract}

\begin{keyword}
linear programming, inverse optimization, robust optimization, parameter estimation
\end{keyword}

\end{frontmatter}

\section{Introduction}

Inverse optimization (IO) aims to determine unspecified parameters of an optimization problem (the \emph{forward problem}) that make a given observed solution optimal. To date, most of the literature has focused on determining parameters of the objective function, assuming the parameters specifying the feasible region are fixed. Depending on whether the given solution is or is not a candidate to be exactly optimal, the corresponding IO models either aim to satisfy the optimality conditions exactly (e.g., \citet{ahuja2001,iyengar2005}) or minimize a measure of suboptimality (e.g., \citet{keshavarz2011,chan2014,bertsimas2015,aswani2018,esfahani2018,chan2018trade,chan2019}), respectively. The two papers in the former category implicitly assume that we have pristine observations of an exactly optimal solution and the feasible set from which it was drawn. In contrast, the papers in the latter category allow that the observations are noisy, that the decision maker suffered from implementation error, or that the assumed forward problem is a lower-dimensional or lower-complexity approximation of the true forward problem.

Methods to impute parameters defining the feasible region, in addition to the objective function, are receiving increasing attention. A key challenge with simultaneously imputing constraint and objective function parameters is that the resulting IO models are generally nonconvex. The vast majority of these papers consider imputing only the ``right-hand-side'' of a linear optimization problem \citep{dempe2006,guler2010,chow2012,cerny2016,saez2016,saez2018,xu2018,lu2019}. Few papers have addressed the problem of imputing the unspecified parameters of the ``left-hand-side'' coefficient matrix, and those that do, exploit specific problem characteristics or assumptions to derive a tractable problem. For example, \citet{birge2017} assume partial access to both the primal and dual solutions to eliminate bilinearities in the IO problem. \citet{brucker2009} exploit the fact that the necessary and sufficient optimality conditions for their forward problem, the minimax lateness scheduling problem, can be written linearly in the parameters to be recovered.

In this paper, we investigate the IO problem of imputing left-hand-side constraint coefficients for a general (without special structure) linear optimization problem, such that an observed solution is minimally suboptimal with respect to some nonzero cost vector. Our motivation is to develop new methodology that extends the small body of work that has been completed on this problem so far. We then extend our analysis to the problem of imputing uncertainty set parameters which appear as left-hand-side coefficients in robust linear optimization problems, for two specific cases where the robust counterpart remains linear: interval uncertainty \citep{bental2000} and cardinality constrained uncertainty \citep{bertsimas2004}.

Although the robust problem can be viewed as a variant of the general problem, there is an independent motivation for studying inverse robust linear optimization. Viewed through a non-robust lens, an observed solution that is an interior point of a fixed feasible set is not a candidate to be optimal and an IO model must therefore minimize suboptimality to fit a cost vector to the data. However, if we view the forward problem as a robust optimization problem, there may be a choice of an uncertainty set, along with a cost vector, that can minimize suboptimality even further for the given interior point solution. In other words, the solution's apparent degree of suboptimality may be large because we have failed to account for uncertainty that the decision maker incorporated into her decision-making process. Such uncertainty effectively shrinks the feasible set such that the observed solution is closer to the boundary of the true, unobservable feasible region. Given the growing adoption of robust optimization in both the research and practitioner communities \citep{bertsimas2011,gabrel2014}, it may increasingly be the case that robustly optimized decisions are observed in a variety of settings and there may be a need for IO models capable of taking such observations as input. To our knowledge, only \citet{chassein2018} consider IO to recover parameters of the uncertainty set for a robust optimization problem. However, their forward problem is a binary integer program for which the linear relaxation produces an optimal solution, and which has min-max regret robustness with interval uncertainty only on the cost vector, with a fixed feasible set. Thus their IO models do not recover uncertainty set parameters controlling the size or shape of the feasible set of the forward problem.

For each of the three forward optimization problems we consider, we will formulate and analyze two IO problem variants. The primary model finds parameter values that minimize the suboptimality associated with the given observation, subject to application-driven side constraints such as non-negativity of the parameters. Note that an observed solution for a linear optimization model can always be made optimal with respect to some nonzero cost vector if the solution is on the boundary of the feasible region. Consequently, because this paper focuses on imputing parameters that determine the feasible region, exact optimality can always be guaranteed if the IO model can choose the constraint parameters without restriction. Accordingly, in the special case where there are no side constraints on the parameters to be imputed, we propose a second inverse model. The second model searches among the potentially multiple optimal solutions to choose one that minimizes the norm distance from a ``prior'' estimate of the parameters, akin to the classical approach of~\citet{ahuja2001}.

\subsection{Motivating applications}

We discuss two application settings to motivate the development of our IO models for imputing constraint coefficients. First, several IO models that impute constraint coefficients, as well as some that impute only an objective function, are motivated by applications in electricity markets. In this setting, each market participant solves an optimization problem to determine a bid which they submit to a facilitator, who in turn incorporates the bids in a market clearing optimization problem that determines prices and the consumption or production allocated to each bidder. In this context, the facilitator may impute bidders' right-hand-side constraint parameters, such as bounds on consumption, which can then be used to inform a pricing strategy that aims to maximize profit or control peak demand \citep{saez2016,saez2018,xu2018,lu2019}. Similarly, a bidder may seek to impute several unknown parameters which can be used in the process of deciding her bid. These parameters include 
%rivals' operating costs, which are objective function
cost coefficients of rivals' models \citep{chen2019}; rival bids, which are objective function coefficients in the facilitator's problem \citep{ruiz2013}; and parameters that describe the routing of energy through the network and the capacity of transmission lines, which are left-hand-side constraint parameters in the facilitator's problem \citep{birge2017}.

Second, we consider radiation therapy treatment planning as a motivation for imputing uncertainty set parameters. Radiation therapy is a cancer treatment modality that aims radiation beams from multiple angles at a tumor, with the goal of delivering an appropriate dose to the target while ensuring that neighboring organs only receive a dose within a safe limit. Treatment planners use an optimization model to determine the beam intensities, however, the quality of the treatment plan can be sensitive to uncertainties such as organ motion due to breathing, patient misalignment with the treatment machine, and the depth at which a beam delivers its peak dose \citep{bortfeld2004,lomax2008}. These uncertainties have motivated robust treatment planning models \citep{unkelbach2007,bortfeld2008} which are available in commercial systems \citep{unkelbach2018}. Independent of this recent development, there is increasing interest in knowledge-based planning, in which a database of treatment plans for historical patients is leveraged to more efficiently generate a treatment plan for a new patient \citep{chanyavanich2011,moore2011,wu2011}. As part of knowledge-based planning, IO has been used to impute objective function weights which generate historical treatment plans and which can be reused to quickly design new treatment plans, however, the work in this area has so far only focused on non-robust forward problems \citep{babier2018a,babier2018b,goli2018}. Given the availability of commerical software to generate robustly optimized treatment plans, such plans will gradually become more available for the purpose of knowledge-based planning, and it will be necessary to impute uncertainty set parameters in addition to objective function weights.

\subsection{Organization of paper and overview of main results}

To summarize, this paper presents three different forward problems, each with two inverse problem variants (see Table \ref{tab:models}). Using small numerical examples, the last subsection in each of Sections 2, 3, and 4 provides insight into the geometry associated with the solution of each model. The development of the inverse robust optimization problems are conceptually similar to the nominal inverse linear problem, so we will omit redundant details wherever possible. Unless otherwise indicated, proofs that do not appear in the body of the paper are contained in the online-only supplementary material.

\begin{table}
\begin{center}
\begin{tabular}{lcc}
\toprule
& \multicolumn{2}{c}{IO model variant} \\
\cmidrule{2-3}
Forward problem & DG & SD \\
\midrule
\multirow{2}{*}{Nominal linear optimization} & Section \ref{sec:LpDg} & Section \ref{sec:LpSd} \\
& NLO-DG & NLO-SD \\
\addlinespace[0.2cm]
Robust linear optimization with & Section \ref{sec:IuDg} & Section \ref{sec:IuSD} \\
interval uncertainty & RLO-IU-DG & RLO-IU-SD \\
\addlinespace[0.2cm]
Robust linear optimization with & Section \ref{sec:CcuDg} & Section \ref{sec:CcuSd} \\
cardinality constrained uncertainty & RLO-CCU-DG & RLO-CCU-SD \\
\bottomrule
\end{tabular}
\caption{Structure of paper and model abbreviations. DG means the duality gap is minimized as an objective. SD means strong duality is enforced as a constraint.}
\label{tab:models}
\end{center}
\end{table}

This paper provides the first comprehensive analysis of inverse linear optimization for the recovery of constraint parameters. Although on the surface each inverse problem is nonconvex, we show through algebraic and geometric analysis that each model can be solved by solving at most $2m$ convex optimization problems, where $m$ is the number of constraints in the forward problem. In certain cases, solving the inverse problem can be reduced to solving $m$ linear optimization problems or evaluating closed form expressions. Table \ref{tab:results} summarizes the complexity of the solution approach for each of the six models.

\begin{table}
\begin{center}
\begin{tabular}{lcc}
\toprule
& DG & SD \\
\midrule
NLO & $m$ convex (linear) & Closed form \\
\addlinespace[0.2cm]
RLO-IU & $m$ convex (linear) & $m$ convex (linear) \\
\addlinespace[0.2cm]
\multirow{2}{*}{RLO-CCU} & $\le m$ linear and $m$ convex & \multirow{2}{*}{$\le m$ linear} \\
& ($\le 2m$ linear) & \\
\bottomrule
\end{tabular}
\caption{The number and type of optimization problems that need to be solved to find a solution to each of the six IO models. DG means the duality gap is minimized as an objective, and SD means strong duality is enforced as a constraint. Number of constraints in forward problem is $m$. Parentheses indicate reduction in complexity if associated side constraints are linear (all DG models) or under appropriate norm choice (in RLO-IU-SD).}
\label{tab:results}
\end{center}
\end{table}

\subsection{Notation}

The following notation will be used in the rest of the paper. Let $\beee$ be the vector of all ones. Let $\beee_i$ be the unit vector with $i$-th component equal to 1. Let $\ba_i$ be the $i$-th row of a matrix $\bA$, which has $m$ rows and $n$ columns. If we have a set of vectors with common index but of differing dimensions such as $\balpha_i$ for all $i\in I$, we will sometimes abuse notation and use $\balpha$ to denote the collection of vectors $\{\balpha_i\}_{i\in I}$. In some optimization models, we will be interested in minimizing over vectors $\{\ba_i\}_{i\in I}$ of the same dimension, in which case we may abuse notation and simply refer to the collection of decision vectors using $\bA$. Thus, whenever $\balpha$ or $\bA$ appear as decision variables in an optimization model, we are optimizing over a set of vectors $\{\ba_i\}_{i\in I}$ or $\{\balpha_i\}_{i\in I}$, respectively. We define $\sgn(x) = 1$ if $x \geq 0$ and $-1$ otherwise.

\section{Nominal linear optimization}\label{sec:nominal}

In this section, we consider the general linear optimization problem
\begin{equation}\label{Lp}
\begin{aligned}
\minimize_{\bx} \;\;\; & \sum_{j\in J}c_j x_j \\
\mbox{subject to} \;\;\; & \sum_{j\in J} a_{ij} x_j \geq b_i, \quad \forall i\in I.
\end{aligned}
\end{equation}

Given $\bb$ and an observed solution $\bxh$, the IO problem aims to identify a constraint matrix $\bA$ that minimizes suboptimality of $\bxh$ with respect to the forward problem and some nonzero cost vector. In Section \ref{sec:LpDg} we first consider the problem NLO-DG, which finds constraint parameters that minimize the duality gap, subject to problem-specific side constraints. These side constraints may render it impossible to make the observed solution $\bxh$ exactly optimal. However, if NLO-DG is found to have a zero duality gap, then the observed solution was in fact optimal with respect to some $\bA$ satisfying the side constraints. In this case, if finding a solution close to some prior parameter estimates is desired, one could then improve the solution quality by solving an IO model which minimally perturbs prior estimates of the constraint parameters subject to not only the side constraints, but also the requirement that there exists a nonzero cost vector rendering the observed solution $\bxh$ exactly optimal. We omit discussion of this model because its solution method would be very similar to that of NLO-DG, however, in Section \ref{sec:LpSd} we consider the independently interesting special case NLO-SD, which finds constraint parameters that make the observed solution exactly optimal, but are not required to satisfy any side constraints.

\subsection{Inverse optimization models}\label{sec:LpDg}

Let $\bpi$ be the dual vector associated with the constraints of the forward problem \eqref{Lp}. The following formulation minimizes the duality gap, subject to some convex constraints $\bA \in \bOmega$, while enforcing primal and dual feasibility:
\begin{subequations}\label{Lpin2}
\begin{align}
\textrm{NLO-DG:}\quad \minimize_{\bA,\bc,\bpi} \;\;\; & \sum_{j \in J} c_{j} \hat{x}_{j} - \sum_{i\in I} b_i \pi_i \label{Lpin2dg} \\
\mbox{subject to} \;\;\; & \bA \in \bOmega, \label{Lpin2side} \\
& \sum_{j \in J} a_{ij} \hat{x}_{j} \geq b_i, \quad \forall i\in I, \label{Lpin1p} \\
& \sum_{i\in I}\pi_i = 1, \label{Lpin1norm} \\
& \sum_{i\in I}a_{ij} \pi_i = c_j, \quad \forall j\in J, \label{Lpin1d} \\
& \pi_i \geq 0, \quad \forall i\in I. \label{Lpin1d2}
\end{align}
\end{subequations}

Constraints \eqref{Lpin1p} and \eqref{Lpin1d}-\eqref{Lpin1d2} represent primal feasibility and dual feasibility, respectively. Notice that the dual feasibility constraints can trivially be satisfied by $(\bc,\bpi) = (\bzero,\bzero)$. NLO-DG would then artificially induce a duality gap of zero while only requiring $\bA$ to satisfy the side constraints and primal feasibility, which is insufficient to guarantee $\bxh$ is optimal with respect to some nonzero $\bc$. Accordingly, constraint \eqref{Lpin1norm} is a normalization constraint that prevents $\bpi = \bzero$ from being feasible, and as a byproduct requires $\bc$ to be in the convex hull of $\{\ba_i\}_{i\in I}$. This set of feasible cost vectors may still include $\bc = \bzero$, but whether $\bc = \bzero$ will be optimal depends on the problem data. Furthermore, it is possible that $\ba_i = \bzero$ will be optimal for some $i\in I$, effectively trivializing that constraint of the forward problem. To prevent $\bc = \bzero$ or $\ba_i = \bzero$ for any $i\in I$ from being optimal for NLO-DG, we will make the following assumption:
\begin{assumption}\label{assumption:nominal_trivial_dg}
For all $i\in I$, $b_i > 0$, or $\ba_i \ne \bzero$ for all $\bA\in\bOmega \cap \{\bA \colon \bA\bxh\geq\bb\}$.
\end{assumption}
\noindent
This assumption requires that each constraint of the forward problem satisfy at least one of two conditions: either the right-hand-side coefficient is positive, in which case a trivial left-hand-side vector would render the constraint infeasible, or the side constraints are defined such that a trivial left-hand-side vector cannot simultaneously satisfy both the side constraints and primal feasibility.

The feasibility of NLO-DG is determined by whether or not $\bOmega$ allows for primal feasibility of the forward problem; the only other constraints of NLO-DG are dual feasibility and the normalization of $\bpi$, which can be satisfied by any $\bpi$ in the unit simplex, and the $\bc$ implied in turn by constraint \eqref{Lpin1d}. We omit the proof of this result, which is straightforward to show.
\begin{pro}\label{pro:NLOSDfeas}
NLO-DG is feasible if and only if there exists $\bA \in \bOmega$ such that $\ba_i^\intercal \bxh \geq b_i$ for all $i\in I$.
\end{pro}

As written, NLO-DG is nonconvex: all its constraints are linear except for the convex side constraints \eqref{Lpin2side} and the dual feasibility constraint \eqref{Lpin1d}, which is bilinear in $\bA$ and $\bpi$. Nevertheless, it is possible to develop an efficient solution method.

\begin{thm}
\label{thm:LpMip}
For all $i\in I$, let
\begin{align}
t_i = \min_{\bA} \left\{ \sum_{j\in J} a_{ij} \hat{x}_j - b_i \colon \bA \in \bOmega, \bA\bxh \geq \bb \right\}, \label{LpMipt}
\end{align}
and let $\bA^{(i)}$ be an optimal solution for \eqref{LpMipt}. Let $i^* \in \argmin_{i\in I} \{t_i\}$, and let $\bA^* = \bA^{(i^*)}$. Then the optimal value of NLO-DG is $t_{i^*}$, and an optimal solution $(\bA,\bc,\bpi)$ is
\begin{align}
\ba_i &= \ba^*_i, \quad \forall i\in I, \label{LpMipopt1} \\
\bc &= \ba^*_{i^*}, \label{LpMipopt2} \\
\bpi &= \beee_{i^*},
\end{align}
where, given Assumption \ref{assumption:nominal_trivial_dg}, $\bc \ne \bzero$ and $\ba_i \ne \bzero$ for all $i\in I$.
\end{thm}

\begin{rmk}\label{rmk:LpMip}
Theorem \ref{thm:LpMip} shows that an optimal solution to the nonconvex inverse problem NLO-DG can be found by solving $m$ convex optimization problems, which become linear whenever the constraints $\bA \in \bOmega$ can be written linearly.
\end{rmk}

\begin{proof}
Substituting \eqref{Lpin1d} into the objective function \eqref{Lpin2dg}, we get the problem
\begin{equation}
\begin{aligned}\label{Lpin2eq1}
\minimize_{\bA,\bpi} \;\;\; & \sum_{i\in I} \pi_i \left(\sum_{j \in J} a_{ij} \hat{x}_j - b_i\right) \\
\mbox{subject to} \;\;\; & \bA \in \bOmega, \; \bA\bxh \geq \bb, \\
& \beee^\intercal\bpi = 1, \; \bpi \geq \bzero.
\end{aligned}
\end{equation}
For a given feasible $\bA$, it is clear that an optimal $\bpi$ is $\beee_{i^*}$, where $i^* \in \argmin_{i\in I} \{\sum_{j\in J}a_{ij} \hat{x}_j - b_i \}$. Problem \eqref{Lpin2eq1} is therefore equivalent to $\min_{i\in I} \left\{ \min_{\bA} \left\{ \sum_{j\in J}a_{ij} \hat{x}_j - b_i \colon \bA \in \bOmega, \bA\bxh \geq \bb \right\} \right\}$. By definition, $\bA^{(i)}$ is an optimal solution for the inner problem, and the optimal value of the outer problem is $\min_{i\in I} \{t_i \}$. Finally, $\bpi = \beee_{i^*}$ and \eqref{Lpin1d} imply that $\bc = \ba^*_{i^*}$.

By Assumption \ref{assumption:nominal_trivial_dg}, $\ba_k = \bzero$ is infeasible for problem \eqref{LpMipt} for all $k\in I, i\in I$, and therefore the optimal solution \eqref{LpMipopt1}-\eqref{LpMipopt2} satisfies $\bc \ne \bzero$ and $\ba_i \ne \bzero$ for all $i\in I$.
\end{proof}

Theorem \ref{thm:LpMip} and its proof can be interpreted as follows. For all $i\in I$, $t_i$ is the minimum achievable surplus for constraint $i$, while respecting primal feasibility and the constraints $\bA \in \bOmega$. Because of the normalization constraint \eqref{Lpin1norm}, the duality gap is equal to a convex combination of the surpluses of the constraints of the forward problem. The minimum possible duality gap will therefore equal the surplus of some constraint $i^*$, and the optimal choice of this constraint is the one with the minimum possible surplus, i.e., $i^* \in \argmin_{i\in I}\{t_i\}$. The constraint vectors are then chosen such that the surplus of constraint $i^*$ equals $t_{i^*}$, and the cost vector is set perpendicular to constraint $i^*$.

\subsubsection{Enforcing strong duality}\label{sec:LpSd}

In this section, we propose an alternative IO model that can be used when there are no side constraints on $\bA$, in which case it may be possible achieve strong duality exactly. In this case, we let $\bah_i$ be given for all $i\in I$, and consider a model variant that minimizes the weighted deviations of the vectors $\ba_i$ from $\bah_i$, while enforcing strong duality, primal and dual feasibility, and the same normalization constraint as in NLO-DG:
\begin{subequations}\label{Lpin1}
\begin{align}
\textrm{NLO-SD:}\quad \minimize_{\bA,\bc,\bpi} \;\;\; & \sum_{i\in I}\xi_i \lVert\ba_i - \bah_i \rVert \label{Lpin1obj} \\
\mbox{subject to} \;\;\; & \sum_{j \in J} c_{j} \hat{x}_{j} - \sum_{i\in I} b_i \pi_i = 0, \label{Lpin1sd} \\
& \eqref{Lpin1p} - \eqref{Lpin1d2}.
\end{align}
\end{subequations}
In the objective function \eqref{Lpin1obj}, $\lVert \cdot \rVert$ is an arbitrary norm, and $\bxi$ is a vector of real-valued weights that is user-tunable. The procedure to estimate the prior vectors $\{\bah_i\}_{i\in I}$ will be application-dependent, and the choice of these estimates will help determine which of the multiple possible imputations that satisfy strong duality will be returned. Unlike previous IO approaches that minimize deviation of $\bc$ from some prior $\hat\bc$, we do not include such an objective since our goal is to determine a constraint matrix $\bA$ that makes $\bxh$ optimal. However, because the vector $\bc$ is still unknown, it must be included as a decision variable in the IO model to facilitate imputing the parameters of interest, i.e., to ensure $\hat\bx$ is optimal with respect to \emph{some} cost vector. Constraint \eqref{Lpin1sd} represents strong duality.

In this subsection, we make the following assumption on the problem data to prevent NLO-SD from having a trivial solution:
\begin{assumption}\label{assumption:nominal_trivial}
For all $i\in I$, $b_i \ne 0$ and $\bah_i \ne \bzero$.
\end{assumption}
\noindent
It is reasonable to expect $\bah_i \ne \bzero$ to be satisfied in most applications, but $b_i \ne 0$ may be considered a strong requirement. However, Assumption \ref{assumption:nominal_trivial} will be used as a sufficient rather than necessary condition for NLO-SD to have a non-trivial optimal solution, and therefore there may be no issue even if $b_i = 0$ for some $i\in I$. If this situation does result in a trivial solution, there are three possible circumventions (see Appendix B in the online supplement for examples). First, we can perturb $b_i$ to be nonzero, although this amounts to a modification of the original problem in which we impute not only $\bA$ but also $b_i$ for at least one constraint. Second, we can perturb $\bah_i$, although we have not characterized the nature of the perturbation necessary to return a non-trivial solution. Third, we can perturb $\bxi$, but this will not work for all problem data (e.g., Example 8 in the online supplement). A more satisfactory solution to this issue is left to future work.

NLO-SD is nonconvex for the same reason as NLO-DG, but the exclusion of constraints $\bA \in \bOmega$ will allow a less complex solution method. First, we note that NLO-SD is always feasible if $\hat\bx \ne \bzero$, and accordingly we make the following assumption for the remainder of this subsection.
\begin{assumption}\label{assumption:lpx}
$\bxh \ne \bzero$.
\end{assumption}
\noindent
This assumption suffices to guarantee feasibility of NLO-SD because if there exists $\hat{j}\in J$ such that $\hat{x}_{\hat{j}} \ne 0$, then the following is a feasible solution to NLO-SD:
\begin{align*}
\bpi &= \beee_{\hat{i}}, \quad \text{for some } \hat{i}\in I, \\
a_{ij} &= \left\{
\begin{array}{ll}
\frac{b_i}{\hat{x}_j} & \text{if } j = \hat{j}, \\
0 & \text{otherwise}, \quad \forall i\in I,
\end{array}
\right. \\
\bc &= \ba_{\hat{i}}.
\end{align*}

Next, we show that the constraints of NLO-SD effectively formalize the geometric intuition that an optimal solution for a linear program must be on the boundary of the feasible region.
\begin{lem}
\label{lem:LpFeas2}
Every feasible solution for NLO-SD satisfies
\begin{subequations}\label{Lpin1abceq}
\begin{align}
& \sum_{j\in J}a_{\hat{i}j} \hat{x}_j = b_{\hat{i}}, \quad \text{for some } \hat{i}\in I, \label{thm:Lp2algc1} \\
& \sum_{j\in J}a_{ij} \hat{x}_j \geq b_i, \quad \forall i\in I. \label{thm:Lp2algc2}
\end{align}
\end{subequations}
Conversely, for every $\bA$ satisfying \eqref{Lpin1abceq}, there exists $(\bc,\bpi)$ such that $(\bA,\bc,\bpi)$ is feasible for NLO-SD.
\end{lem}

Lemma \ref{lem:LpFeas2} allows us to characterize an optimal solution for NLO-SD and suggests an efficient solution method.

\begin{thm}\label{thm:LpAlg}
For all $i\in I$, let
\begin{align}
f_i & = \frac{\xi_i \left\lvert \bah_i^\intercal \bxh - b_i \right\rvert} {\lVert \bxh \rVert^*}, \label{Lp2act}
 \\
\ba^f_i & = \bah_i - \frac{\bah_i^\intercal \bxh - b_i}{\lVert \bxh \rVert^*}\bv(\bxh), \label{Lp2actsol}
 \\
g_i & = \label{Lp2feas}
	\begin{cases}
		f_i & \text{if } \bah_i^\intercal\bxh < b_i, \\
		0 & \text{otherwise},
	\end{cases} \\
\ba^g_i & = \label{Lp2feassol}
	\begin{cases}
		\ba^f_i & \text{if } \bah_i^\intercal\bxh < b_i, \\
		\bah_i & \text{otherwise},
	\end{cases}
\end{align}
where $\lVert \bxh \rVert^* = \max_{\lVert \bv \rVert = 1} \bxh^\intercal \bv$ is the dual norm of $\lVert \cdot \rVert$, and $\bv(\bxh) \in \argmax_{\lVert \bv \rVert = 1} \bxh^\intercal \bv$. Let $i^* \in \argmin_{i\in I} \{f_i - g_i \}$. Then the optimal value of NLO-SD is $f_{i^*} + \sum_{i\in I\setminus\{i^*\}}g_i$, and an optimal solution $(\bA, \bc, \bpi)$ is
\begin{align}
& \ba_i = \left\{
\begin{array}{ll}
\ba^f_i & \text{if } i=i^*, \label{Lp2actsol2}
 \\
\ba^g_i & \text{if } i\in I\setminus\{i^*\}, %\label{Iu2feassol2}
\end{array}
\right. \\
& \bc = \ba_{i^*}, \label{Lp2csol} \\
& \bpi = \beee_{i^*},
\end{align}
where, given Assumption \ref{assumption:nominal_trivial}, $\bc \ne \bzero$ and $\ba_i \ne \bzero$ for all $i\in I$.
\end{thm}

\begin{rmk}\label{rmk:LpAlg}
Theorem \ref{thm:LpAlg} shows that an optimal solution to the nonconvex inverse optimization problem NLO-SD can be found in closed form.
\end{rmk}

\begin{proof}
By Lemma \ref{lem:LpFeas2}, solving NLO-SD is equivalent to solving the following optimization problem for all $\hat{i}\in I$, and taking the minimum over all $\lvert I\rvert$ optimal values:
\begin{equation}\label{Lpin1eq}
\begin{aligned}
\minimize_{\bA} \;\;\; & \sum_{i\in I}\xi_i \lVert \ba_i - \bah_i \rVert \\
\mbox{subject to} \;\;\; & \sum_{j\in J}a_{\hat{i}j} \hat{x}_j = b_{\hat{i}}, \\
& \sum_{j\in J}a_{ij} \hat{x}_j \geq b_i, \quad \forall i\in I.
\end{aligned}
\end{equation}
Suppose we fix some $\hat{i} \in I$. Since formulation \eqref{Lpin1eq} is separable by $i$, the optimal value of the $\hat{i}$-th formulation~\eqref{Lpin1eq} is $\bar{f}_{\hat{i}} + \sum_{i\in I\setminus\{\hat{i}\}} \bar{g}_i$, where we let
\begin{align}
\bar{f}_i & = \min_{\ba_i}\left\{\xi_i \lVert\ba_i - \bah_i \rVert \colon \sum_{j\in J}a_{ij} \hat{x}_j = b_i \right\}, \label{Lp2actbar} \\
\bar{g}_i & = \min_{\ba_i}\left\{\xi_i \lVert\ba_i - \bah_i \rVert \colon \sum_{j\in J}a_{ij} \hat{x}_j \geq b_i \right\}, \label{Lp2feasbar}
\end{align}
for all $i\in I$. Because problem \eqref{Lp2actbar} is the projection of a point $\bah_i$ onto the hyperplane $\bxh^\intercal \ba_i = b_i$, it can be shown using Theorem 2.1 of \citet{mangasarian1999} that its optimal solution is $\ba^f_i$ and its optimal value is $\bar{f}_i = f_i$. Problem \eqref{Lp2feasbar} is the projection of $\bah_i$ onto the closed half-space $\bxh^\intercal \ba_i \geq b_i$, so its optimal solution consists of two cases: if $\bah_i$ is in that closed half-space then it is optimal, otherwise its projection onto the closed half-space must be on its boundary, i.e., it equals $\ba^f_i$. Accordingly, the optimal value of problem \eqref{Lp2feasbar} is $\bar{g}_i = g_i$.

Therefore, the optimal value of NLO-SD is
\begin{align*}
\min_{\hat{i}\in I} \left\{ f_{\hat{i}} + \sum_{i\in I\setminus\{\hat{i}\}}g_i \right\}.
\end{align*}
Clearly, the optimal index $i^*$ must satisfy $i^* \in \argmin_{i\in I}\{f_i - g_i \}$. An optimal $\bA$ is given by \eqref{Lp2actsol2}, which is derived from the optimal solutions of \eqref{Lp2actbar} and \eqref{Lp2feasbar}, and an optimal cost vector is $\bc = \ba_{i^*}$.

For all $i\in I$, $\ba^f_i \ne \bzero$ since $\ba_i = \bzero$ is infeasible for problem \eqref{Lp2actbar} due to the assumption that $b_i \ne 0$. For all $i\in I$, we also have $\ba^g_i \ne \bzero$ due to the assumption that $\bah_i \ne \bzero$. Therefore the optimal $\ba_i \ne \bzero$ for all $i\in I$, and the optimal $\bc \ne \bzero$.
\end{proof}

Theorem \ref{thm:LpAlg} can be interpreted as follows. For all $i\in I$, $f_i$ is the minimal value of the $i$-th term in objective function \eqref{Lpin1obj} such that constraint $i$ of the forward problem is rendered active. Similarly, $g_i$ is the minimal value for constraint $i$ to be rendered feasible; clearly, $g_i \ne 0$ only if $\bxh$ is infeasible with respect to $\bah_i$. For $\bxh$ to be optimal for the forward problem, some constraint $i^*$ must have $\ba_{i^*}$ set such that $\bxh$ is on the boundary. The optimal choice of this constraint is the one that requires the minimal additional increase in $\xi_i \lVert \ba_i - \bah_i \rVert$ for the constraint to be active rather than merely feasible, i.e., $i^* \in \argmin \{ f_i - g_i \}$. To satisfy the optimality conditions, the cost vector is set perpendicular to this active constraint.

Theorem \ref{thm:LpAlg} also draws a close parallel with one of the main results from \citet{chan2019}. There, the focus is on imputing a cost vector for a linear optimization problem, given a fixed feasible region and an observed interior point $\bxh$. It was shown that an optimal solution involves projecting $\bxh$ to the boundary $\ba^\intercal_i \bx = b_i$ of each constraint, identifying the constraint $i^*$ associated with the minimal distance, and then setting the cost vector perpendicular to that constraint. Similarly, an optimal solution to NLO-SD involves projecting $\bah_i$ to the hyperplane $\ba_i^\intercal \bxh = b_i$ for each constraint, identifying the constraint $i^*$ associated with the minimal distance, and then setting the cost vector perpendicular to that constraint. In the process, we also adjust the normal vector of the constraint $i^*$ such that the constraint is active with respect to $\bxh$. In contrast, the constraints' normal vectors are all given and fixed in \citet{chan2019}.

\subsection{Numerical examples}\label{sec:LpEx}

In this section, we provide numerical examples that illustrate solutions to NLO-DG and NLO-SD and their associated geometric characteristics.

\begin{exmp}[\textbf{NLO-DG}]\label{ex:LpDg}
Let $\bxh = (-2,6)$, $\bb = (-6,-6,-10)$, and
\begin{align*}
\bOmega = \{ & \bA \colon 1 \le a_{11} \le 1.5, \; 2 \le a_{22} \le 3, \\
& a_{12} = 0, \; a_{21} = 0, \\
& a_{31} \le -2, \; -2 \le a_{32} \le -0.5, \\
& a_{31} + 2a_{22} \le 2 \}.
\end{align*}
It is easy to check that there exists $\bA \in \bOmega$ such that $\bA\bxh \geq \bb$, so NLO-DG is feasible by Proposition~\ref{pro:NLOSDfeas}. Applying Theorem \ref{thm:LpMip} and solving formulation~\eqref{LpMipt} for $i = 1, 2, 3$, we compute that $\bt = (3, 18, 2)$, hence $i^* = 3$. An optimal solution of formulation~\eqref{LpMipt} corresponding to $i^* = 3$ is
\begin{align*}
\bA^* =
\begin{pmatrix}
1 & 0 \\
0 & 2 \\
-2 & -2
\end{pmatrix},
\end{align*}
and the optimal cost vector is $\bc = \ba^*_3 = (-2, -2)$. These results are illustrated in Figure \ref{fig:LpDg}. The observed solution $\bxh$ is an interior point of the imputed feasible region because the constraints $\bA \in \bOmega$ do not admit a feasible region that puts the observed solution on its boundary. The IO model NLO-DG instead minimizes the surplus of a single constraint, thereby minimizing the duality gap by setting the cost vector perpendicular to this constraint.
\end{exmp}

\begin{figure*}
\centering
\begin{subfigure}{0.5\textwidth}
\centering
\begin{tikzpicture}
		\begin{axis}[
			axis x line=center,
			axis y line=center,
			x=0.3cm, y=0.3cm, %sets the axis units, effectively controlling size of picture
			xlabel={$x_1$},
			ylabel={$x_2$},
			xlabel style={right},
			ylabel style={above},
			xtick={-8,-4,4,8},
			ytick={-4,4,8,12,16,20,24},
			minor xtick={-8,-6,-4,...,8}, %for grid
			minor ytick={-6,-4,...,26}, %for grid
			grid=minor,
			legend cell align={left},
			xmin=-8.9,
			xmax=8.9,
			ymin=-6.9,
			ymax=26.9]
				
	  \filldraw (-2,6) circle (1.5pt) node[above] {$\bxh$}; % observed solution		
		%\draw[->] (0,0) -- (-2,6) node[above] {$\bxh$}; % observed solution		
		\draw[semithick,blue] (-6,-3) -- (-6,11) -- (8,-3) -- (-6,-3); % imputed feasible region
		\draw[->, semithick, blue] (2.5,2.5) -- (0.5,0.5) node[above] {$\bc$}; % cost vector
		
		% Legend
		\addlegendimage{semithick,blue}
		\addlegendentry{Imputed feasible region}
		
		\end{axis}
\end{tikzpicture}
\caption{Example \ref{ex:LpDg}: NLO-DG.}
\label{fig:LpDg}
\end{subfigure}~~~~ 
\begin{subfigure}{0.5\textwidth}
\centering
\begin{tikzpicture}
		\begin{axis}[
			axis x line=center,
			axis y line=center,
			x=0.3cm, y=0.3cm, %sets the axis units, effectively controlling size of picture
			xlabel={$x_1$},
			ylabel={$x_2$},
			xlabel style={right},
			ylabel style={above},
			xtick={-8,-4,4,8},
			ytick={-4,4,8,12,16,20,24},
			minor xtick={-8,-6,-4,...,8}, %for grid
			minor ytick={-6,-4,...,26}, %for grid
			grid=minor,
			legend cell align={left}, 
			xmin=-8.9,
			xmax=8.9,
			ymin=-6.9,
			ymax=26.9]
				
	  \filldraw (-2,6) circle (1.5pt) node[above] {$\bxh$}; % observed solution		
		
		\draw[semithick,blue] (-8,-6) -- (0,10) -- (8,-6) -- (-8,-6); % imputed feasible region
		\draw[very thick,red,dashed] (-6,22) -- (8,-6) -- (-6,-6) -- (-6,22); % prior feasible region
		
		\draw[->, semithick, blue] (-3,4) -- (-1.8,3.4) node[below] {$\bc$}; % cost vector
				
		\addlegendimage{very thick,red,dashed}
		\addlegendentry{Prior feasible region}
		\addlegendimage{semithick,blue}
		\addlegendentry{Imputed feasible region}
		
		\end{axis}
\end{tikzpicture}
\caption{Example \ref{ex:LpSd}: NLO-SD.}
\label{fig:LpSd}
\end{subfigure}
\caption{Numerical examples for the nominal IO models.}
\label{fig:LpEx}
\end{figure*}
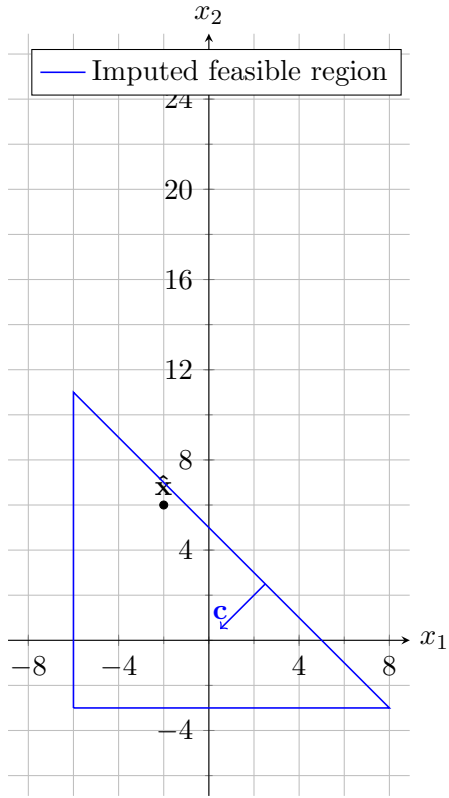
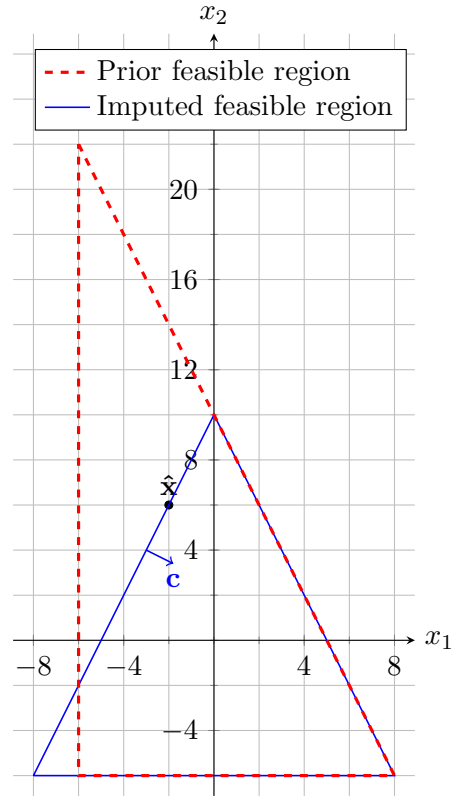

\begin{exmp}[\textbf{NLO-SD}]\label{ex:LpSd}
Let the norm in the objective function \eqref{Lpin1obj} be the Euclidean norm and $\bxi = \beee$ for simplicity. Let the observed solution be $\bxh = (-2, 6)$ and the remaining problem data be
\begin{align*}
\bAh =
\begin{pmatrix}
1 & 0 \\
0 & 1 \\
-2 & -1
\end{pmatrix}, \;
\bb =
\begin{pmatrix}
-6 \\
-6 \\
-10
\end{pmatrix}.
\end{align*}
The prior feasible region defined by $(\bAh, \bb)$ is shown in Figure \ref{fig:LpSd}. We find a solution by applying Theorem \ref{thm:LpAlg}. Since $\bxh$ is feasible with respect to $(\bAh,\bb)$, $g_i = 0$ for all $i\in I$. Evaluating \eqref{Lp2act} for each $i \in I$, we get $\bff = (0.63, 1.90, 1.26)$, which means that $i^* = 1$. As a result, the first constraint will be adjusted so that $\ba_1 = \ba^f_1 = (1.2,-6)$, while the other two constraints are unchanged. In other words, the optimal solution of the IO model NLO-SD only needs to adjust a single constraint to put $\hat\bx$ on the boundary of the imputed feasible region, which is possible because there are no inhibitory side constraints on $\bA$. The cost vector that makes $\hat\bx$ optimal is perpendicular to the first constraint, as shown in Figure \ref{fig:LpSd}.
\end{exmp}

\section{Robust linear optimization with interval uncertainty}\label{sec:iu}

In this section, we consider a robust linear optimization problem with interval uncertainty. Let $J_i \subseteq J$ index the coefficients in the $i$-th row of $\bA$ that are subject to interval uncertainty, which is defined by the parameters $\balpha_i$, which are given for all $i \in I$. Let $\ba_{i}$ and $b_i$ also be given for all $i \in I$. The robust problem is:
\begin{equation}\label{Iu}
\begin{aligned}
\minimize_{\bx} \;\;\; & \sum_{j\in J}c_j x_j \\
\mbox{subject to} \;\;\; & \sum_{j\in J_i} \tilde a_{ij} x_j + \sum_{j\in J\setminus J_i} a_{ij} x_j \geq b_i, \quad \forall \tilde a_{ij} \in [a_{ij} - \alpha_{ij}, a_{ij} + \alpha_{ij}], i\in I.
\end{aligned}
\end{equation}
Each constraint $i\in I$ can be written as $\sum_{j\in J}a_{ij} x_j - \sum_{j\in J_i} \alpha_{ij}|x_j| \geq b_i$. Equivalently this can be written as $\sum_{j\in J} \bar{a}_{ij}(\balpha_i, \bx) x_j \geq b_i$, where
\begin{align}
\label{eq:iuabar}
\bar{a}_{ij}(\balpha_i, \bx) & = \left\{
\begin{array}{ll}
a_{ij} - \sgn({x_j})\alpha_{ij} & \text{if } j\in J_i, \\
a_{ij} & \text{if } j\in J\setminus J_i,
\end{array}
\right.  \forall \balpha_i \geq \bzero, i\in I, \bx\in\mathbb{R}^n.
\end{align}
Formulation~\eqref{Iu} can be linearized as \citep{bental2000}:
\begin{subequations}\label{Iu2}
\begin{align}
\minimize_{\bx,\bu} \;\;\; & \sum_{j \in J} c_j x_j \\
\mbox{subject to} \;\;\; & {\alpha}_{ij}x_j + u_{ij} \geq 0, \quad \forall j\in J_i, i\in I, \label{Iu2a} \\
& -{\alpha}_{ij}x_j + u_{ij} \geq 0, \quad \forall j\in J_i, i\in I,  \\
& \sum_{j \in J} a_{ij} x_{j} - \sum_{j\in J_i}u_{ij} \geq b_i, \quad \forall i\in I. \label{Iu2c}
\end{align}
\end{subequations}

The corresponding inverse problem of formulation~\eqref{Iu2} aims to impute the $\balpha$ parameters, given the nominal constraint matrix $\bA$. Viewed through the lens of Section~\ref{sec:nominal}, this can be thought of as a special case of the recovery of constraint parameters for a non-robust linear optimization problem. However, the inverse problem here is additionally constrained by the requirement that $\balpha \geq \bzero$ and that many of the constraint coefficients in formulation~\eqref{Iu2} are fixed. Moreover, the general inverse linear optimization method will not by itself recognize that most constraints of \eqref{Iu2} are auxiliary, and may possibly set the cost vector perpendicular to an auxiliary constraint. The general method will also not necessarily prevent trivial solutions: the realization of a constraint's left-hand-side coefficients depends on the orthant containing a given solution, as indicated by \eqref{eq:iuabar}, and therefore an IO method that does not take account of the forward problem's robust structure ignores the possibility that a constraint's realization is trivial in orthants aside from the one containing the observed solution. All of these differences from the IO problems in Section~\ref{sec:nominal} can be accommodated by modifying NLO-DG and NLO-SD and their solution methods. However, we instead develop methods specifically addressing the robust formulation~\eqref{Iu2} to yield more precise insights.

Given $\ba_i, b_i$ and $J_i$ for all $i \in I$, and a feasible $\hat\bx$ for the nominal problem (i.e., formulation~\eqref{Iu} with $J_i = \varnothing$ for all $i$), the goal of the IO problem is to determine nonnegative parameters $\balpha_i$ for all $i \in I$ defining the uncertainty set, such that $\bxh$ is minimally suboptimal for some nonzero cost vector. Without loss of generality, we assume that every row has at least one coefficient that is subject to uncertainty (if we did not make this assumption, we would define $\hat{I} := \{i\in I \colon J_i \ne \varnothing\}$ and replace $I$ with $\hat{I}$ throughout the following development where appropriate):

\begin{assumption}\label{assumption:Iu1}
$J_i \ne \varnothing, \; \forall i\in I$.
\end{assumption}

In the context of the IO problem that corresponds to the forward problem $\eqref{Iu}$, a trivial solution is one in which either $\bc = \bzero$ or $\bab_i(\balpha_i, \bx) = \bzero$ for some $i\in I, \bx\in\mathbb{R}^n$. We will make an additional assumption to prevent the IO models from returning trivial solutions:
\begin{assumption}\label{assumption:iu_trivial}
For all $i\in I$, $b_i > 0$ or $a_{ij} \ne 0$ for some $j\in J\setminus J_i$.
\end{assumption}
\noindent
This assumption requires that for each constraint of the forward problem, either the right-hand-side coefficient is positive (as in Assumption \ref{assumption:nominal_trivial_dg}), or one of the left-hand-side coefficients is known with certainty to be nonzero.

As before, we consider two variants of this problem: the first IO model in Section \ref{sec:IuDg} minimizes the duality gap, whereas the IO model in Section \ref{sec:IuSD} assumes a zero duality gap. Section \ref{sec:IuEx} provides numerical examples.

\subsection{Inverse optimization models}\label{sec:IuDg}

Let $\lambda_{ij}, \mu_{ij}, \pi_i$ be the dual variables corresponding to constraints \eqref{Iu2a}-\eqref{Iu2c}, respectively. The following formulation minimizes the duality gap, subject to convex constraints $\balpha \in \bOmega$, while enforcing primal and dual feasibility. Given Assumption \ref{assumption:Iu1}, the vectors $\balpha_i$ all have dimension at least one.
\begin{subequations}\label{Iuin2}
\begin{align}
\textrm{RLO-IU-DG:}\quad \minimize_{\balpha,\bc,\bu,\bpi,\blambda,\bmu} \;\;\; & \sum_{j \in J} c_{j} \hat{x}_{j} - \sum_{i\in I} b_i\pi_i \label{Iuin2dualitygap} \\
\mbox{subject to} \;\;\; & \balpha \in \bOmega, \label{Iuin2side} \\
& {\alpha}_{ij}\hat{x}_j + u_{ij} \geq 0, \quad \forall j\in J_i, i\in I, \label{Iuin1p1} \\
& -{\alpha}_{ij}\hat{x}_j + u_{ij} \geq 0, \quad \forall j\in J_i, i\in I, \label{Iuin1p2} \\
& \sum_{j \in J} a_{ij} \hat{x}_{j} - \sum_{j\in J_i}u_{ij} \geq b_i, \quad \forall i\in I, \label{Iuin1p3} \\
& \alpha_{ij} \geq 0, \quad \forall j\in J_i, i\in I, \label{Iuin1u} \\
& \sum_{i\in I}\pi_i = 1, \label{Iuin1norm} \\
& \sum_{i\in I}a_{ij} \pi_i + \sum_{i\in I \colon j\in J_i}\alpha_{ij}( \lambda_{ij} - {\mu}_{ij}) = c_j, \quad \forall j\in J, \label{Iuin1d1} \\
& \pi_i = {\lambda}_{ij} + {\mu}_{ij}, \quad \forall j\in J_i, i\in I, \label{Iuin1d2} \\
& \pi_i, {\lambda}_{ij}, {\mu}_{ij} \geq 0, \quad \forall j\in J_i, i\in I. \label{Iuin1d3}
\end{align}
\end{subequations}
RLO-IU-DG is constructed in a conceptually similar manner as NLO-DG. Constraints \eqref{Iuin1p1}-\eqref{Iuin1p3} and \eqref{Iuin1d1}-\eqref{Iuin1d3} represent primal feasibility and dual feasibility, respectively. To prevent the trivial solution $(\bc,\bpi) = (\bzero,\bzero)$ from being feasible, we again include the normalization constraint \eqref{Iuin1norm}. All constraints of RLO-IU-DG are linear except for the convex side constraints \eqref{Iuin2side} and the bilinear dual feasibility constraint \eqref{Iuin1d1}, but nevertheless we will be able to determine an efficient solution method.

First, we show that feasibility of RLO-IU-DG is determined by whether or not $\bOmega$ allows for primal feasibility of the forward problem.
\begin{pro}\label{pro:IuTwoFeas}
RLO-IU-DG is feasible if and only if there exists nonnegative $\balpha\in \bOmega$ such that $\sum_{j \in J} a_{ij} \hat{x}_{j} - \sum_{j \in J_i} \alpha_{ij} \lvert\hat{x}_{j}\rvert \geq b_{i}$ for all $i\in I$.
\end{pro}

Building on Proposition~\ref{pro:IuTwoFeas}, an analogous result to Theorem~\ref{thm:LpMip} can be derived, which means that RLO-IU-DG can be solved by solving $m$ convex optimization problems. The interpretation of the following theorem is conceptually identical to the interpretation of Theorem \ref{thm:LpMip}. The proofs of the two results both involve substituting the dual feasibility constraint with $\bc$ into the objective function to show that there exists an optimal solution with binary $\bpi$. In this case, however, the presence of additional dual variables corresponding to the auxiliary constraints of the robust forward problem necessitates additional algebraic analysis to arrive at  a similar result.

\begin{thm}
\label{thm:IuDg}
For all $\hat{i}\in I$, let $t_{\hat{i}}$ be the optimal value and let $\balpha^{(\hat{i})}$ be an optimal solution for the problem
\begin{equation}
\begin{aligned}\label{IuMipt}
\minimize_{\balpha} \;\;\; & \sum_{j\in J} a_{\hat{i}j} \hat{x}_j - \sum_{j\in J_{\hat{i}}}\alpha_{\hat{i}j}|\hat{x}_j| - b_{\hat{i}} \\
\mbox{\emph{subject to}} \;\;\; & \sum_{j\in J} a_{ij} \hat{x}_j - \sum_{j\in J_i}\alpha_{ij}|\hat{x}_j| \geq b_i, \quad \forall i\in I, \\
& \balpha \in \bOmega, \balpha \geq \bzero.
\end{aligned}
\end{equation}
Let $i^* \in \argmin_{\hat{i}\in I} \{t_{\hat{i}}\}$, and let $\balpha^* = \balpha^{(i^*)}$. Then the optimal value of RLO-IU-DG is $t_{i^*}$ and there exists an optimal solution with
\begin{align}
\balpha_i &= \balpha^*_i, \quad \forall i\in I, \label{IuSdopt1} \\
\bc &= \bab_{i^*}(\balpha^*_{i^*}, \bxh), \label{IuSdopt2} \\
\bpi &= \beee_{i^*}, \label{IuSdopt3}
\end{align}
where, given Assumption \ref{assumption:iu_trivial}, $\bc \ne \bzero$ and $\bab_i(\balpha_i,\bx) \ne \bzero$ for all $i\in I, \bx\in\mathbb{R}^n$.
\end{thm}

\begin{rmk}\label{rmk:IuDg}
Theorem \ref{thm:IuDg} shows that an optimal solution to the nonconvex inverse problem RLO-IU-DG can be found by solving $m$ convex optimization problems, which become linear whenever the constraints $\balpha \in \bOmega$ can be written linearly.
\end{rmk}

Although both Theorems \ref{thm:LpMip} and \ref{thm:IuDg} set the cost vector perpendicular to the constraint with the minimum surplus, a difference arises in the latter case due to the structure of the robust constraints. As shown by \eqref{eq:iuabar}, the vector perpendicular to a robust constraint changes as the constraint crosses into different orthants (see Figure \ref{fig:IuEx}). Therefore equation \eqref{IuSdopt2} more specifically sets the cost vector perpendicular to the part of constraint $i^*$ that is contained in the same orthant as $\bxh$.

\subsubsection{Enforcing strong duality}\label{sec:IuSD}

As in Section \ref{sec:LpSd}, we propose an alternative IO model that minimizes the weighted deviation of the uncertainty set parameters $\balpha_i$ from given prior values $\balphah_i$ while enforcing strong duality, and primal and dual feasibility, without side constraints on $\balpha$. We make an additional assumption that there is at least one column $j\in J$ that has an uncertain coefficient and $\hat{x}_j \ne 0$:
\begin{assumption}\label{assumption:Iu2}
There exists some $\hat{i} \in I$ and $j\in J_{\hat{i}}$ such that $\hat{x}_j \ne 0$.
\end{assumption}
\noindent
This assumption is slightly stronger than the assumption $\bxh\ne\bzero$ that we made to guarantee feasibility of NLO-SD, and is needed in order for the values of the unknown parameters to affect the optimality of the observed solution: without this assumption, all $\alpha_{ij}$ would be multiplied by zero and modifying $\balpha$ would not change the surplus of any constraint with respect to $\bxh$.

We propose the following IO model. Given Assumption \ref{assumption:Iu1}, the vectors $\balpha_i$ in the objective function all have dimension at least one.
\begin{subequations}\label{Iuin1}
\begin{align}
\textrm{RLO-IU-SD:}\quad \minimize_{\balpha,\bc,\bu,\bpi,\blambda,\bmu} \;\;\; & \sum_{i\in I}\xi_i \lVert\balpha_i - \balphah_i \rVert \label{Iuin1obj} \\
\mbox{subject to} \;\;\; & \sum_{j \in J} c_{j} \hat{x}_{j} - \sum_{i\in I} b_i \pi_i = 0, \label{Iuin1sd} \\
& \eqref{Iuin1p1}-\eqref{Iuin1d3}.
\end{align}
\end{subequations}
In the objective function \eqref{Iuin1obj}, $\lVert \cdot \rVert$ is an arbitrary norm. Constraint \eqref{Iuin1sd} represents strong duality.

First, we show that feasibility of RLO-IU-SD is entirely determined by feasibility of $\bxh$ with respect to the nominal problem.
\begin{pro}
\label{pro:IuFeas}
Given Assumption \ref{assumption:Iu2}, RLO-IU-SD is feasible if and only if $\sum_{j \in J} a_{ij} \hat{x}_{j} \geq b_{i}$ for all $i\in I$.
\end{pro}
\noindent
The geometric intuition underlying Proposition \ref{pro:IuFeas} is twofold: the robust feasible region is a subset of the nominal feasible region for any choice of $\balpha$, and $\bxh$ must lie on the boundary of the robust feasible region in order to be optimal. Hence if $\bxh$ is feasible for the nominal problem, it is possible to set $\balpha$ that shrinks the feasible region such that some constraint is active at $\bxh$. And conversely, if $\bxh$ is not feasible for the nominal problem, then there is no way to grow the feasible region such that $\bxh$ lies on the boundary, or is even feasible.

We now characterize an optimal solution to RLO-IU-SD and devise an efficient solution method reflecting the same geometric intuition underlying Proposition \ref{pro:IuFeas}.
\begin{thm}\label{thm:IuAlg}
For all $\hat{i}\in I$, let $t_{\hat{i}}$ be the optimal value and let $\balpha^{(\hat{i})}$ be an optimal solution for the problem
\begin{subequations}\label{Iuin1eq}
\begin{align}
\minimize_{\balpha} \;\;\; & \sum_{i\in I}\xi_i \lVert \balpha_i - \balphah_i \rVert \\
\mbox{\emph{subject to}} \;\;\; & \sum_{j\in J}a_{\hat{i}j} \hat{x}_j -\sum_{j\in J_{\hat{i}}}{{\alpha}_{\hat{i}j}|\hat{x}_{j}|} = b_{\hat{i}} \label{Iuin1eq2} \\ %, \quad \text{for some } \hat{i}\in I
& \sum_{j\in J}a_{ij} \hat{x}_j -\sum_{j\in J_i}{{\alpha}_{ij}|\hat{x}_{j}|} \geq b_i, \quad \forall i\in I, \label{Iuin1eq3} \\
& \alpha_{ij} \geq 0, \quad \forall j\in J_i, i\in I. \label{Iuin1eq4}
\end{align}
\end{subequations}
Let $i^*\in\argmin_{\hat{i}\in I} \{t_{\hat{i}}\}$, and let $\balpha^* = \balpha^{(i^*)}$. Then the optimal value of RLO-IU-SD is $t_{i^*}$, and there exists an optimal solution with $(\balpha, \bc, \bpi)$ as stated in \eqref{IuSdopt1}-\eqref{IuSdopt3}, where, given Assumption \ref{assumption:iu_trivial}, $\bc \ne \bzero$ and $\bab_i(\balpha_i, \bx) \ne \bzero$ for all $i\in I, \bx\in\mathbb{R}^n$.
\end{thm}

\begin{rmk}\label{rmk:IuAlg}
Theorem \ref{thm:IuAlg} shows that an optimal solution to the nonconvex inverse problem RLO-IU-SD can be found by solving $m$ convex problems (linear with appropriate choice of $\lVert\cdot\rVert$).
\end{rmk}

As in NLO-SD, an optimal solution to RLO-IU-SD requires that at least one constraint of the forward problem be active: for all $\hat{i}\in I$, $t_{\hat{i}}$ is the minimum value of $\sum_{i\in I}\xi_i \lVert \balpha_i - \balphah_i \rVert$ such that constraint $\hat{i}$ is set active. The parameters $\balpha$ are set such that the constraint with the minimum value of $t_{\hat{i}}$ is active and all other constraints are feasible with minimal perturbation to the prior $\balphah$, and the cost vector is set perpendicular to the active constraint. A difference between Theorems \ref{thm:LpAlg} and \ref{thm:IuAlg} is that for the former, the value of $\bA$ can be evaluated as the closed form solution to problems of the form \eqref{Lp2actbar} and \eqref{Lp2feasbar}, but in the latter the value of $\balpha$ has to be obtained as the solution to the auxiliary optimization problem \eqref{Iuin1eq}. Although problem \eqref{Iuin1eq} can be decomposed by $i$ into problems corresponding to \eqref{Lp2actbar} and \eqref{Lp2feasbar}, they would be constrained by $\balpha_i \geq \bzero$ and therefore would not have closed form solutions as projections onto a hyperplane and closed half-space, respectively.

\subsection{Numerical examples}\label{sec:IuEx}

In this section, we give numerical examples to illustrate the geometric characteristics of the solutions for RLO-IU-DG and RLO-IU-SD. These examples demonstrate how the optimal inverse solution is found and how it relates to the geometry of the robust feasible region induced by the uncertainty set parameters.

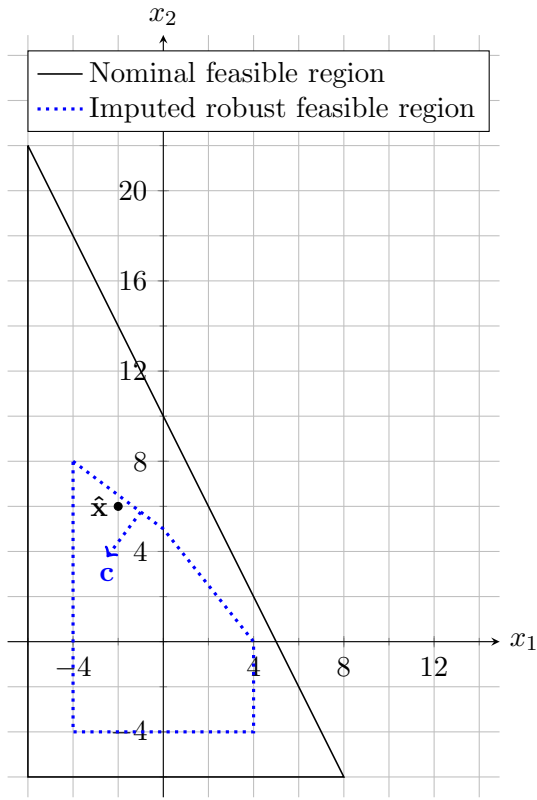
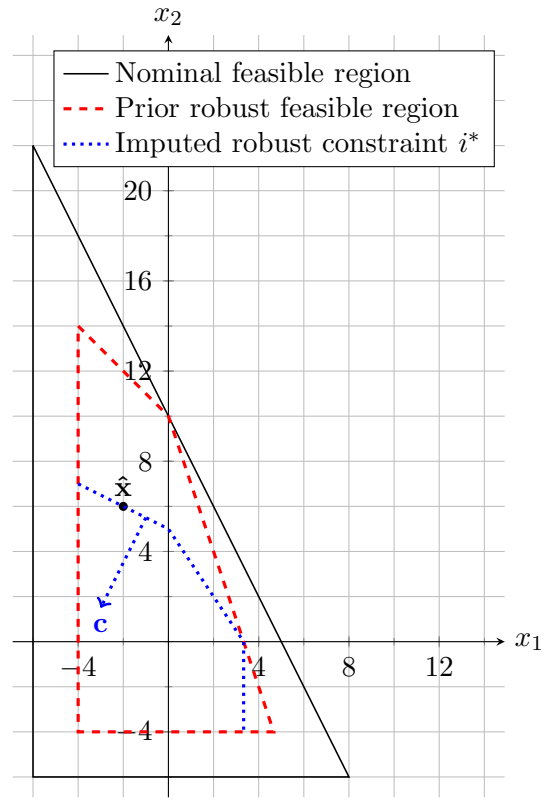
\begin{figure*}
\centering
\begin{subfigure}{0.5\textwidth}
\centering
\begin{tikzpicture}
		\begin{axis}[
			axis x line=center,
			axis y line=center,
			x=0.3cm, y=0.3cm, %sets the axis units, effectively controlling size of picture
			xlabel={$x_1$},
			ylabel={$x_2$},
			xlabel style={right},
			ylabel style={above},
			xtick={-4,4,8,12},
			ytick={-4,4,8,12,16,20,24},
			minor xtick={-8,-6,-4,...,8,10,12,14}, %for grid
			minor ytick={-6,-4,...,26}, %for grid
			grid=minor,
			legend cell align={left}, 
			xmin=-6.9,
			xmax=14.9,
			ymin=-6.9,
			ymax=26.9]
				
	  \filldraw (-2,6) circle (1.5pt) node[left] {$\bxh$}; % observed solution		
		\draw[semithick] (-6,22) -- (8,-6) -- (-6,-6) -- (-6,22); % nominal feasible region
		\draw[very thick,blue,dotted] (-4,8) -- (0,5) -- (4,0) -- (4,-4) -- (-4,-4) -- (-4,8); % robust feasible region
		\draw[->, very thick, blue,dotted] (-1,5+3/4) -- (-2.5,3+3/4) node[below] {$\bc$}; % cost vector
		
		% Legend
		\addlegendimage{semithick}
		\addlegendentry{Nominal feasible region}
		\addlegendimage{very thick,blue,dotted}
		\addlegendentry{Imputed robust feasible region}
		
		\end{axis}
\end{tikzpicture}
\caption{Example \ref{ex:IuDg}: RLO-IU-DG.}
\label{fig:IuDg}
\end{subfigure}~~~~ 
\begin{subfigure}{0.5\textwidth}
\centering
\begin{tikzpicture}
		\begin{axis}[
			axis x line=center,
			axis y line=center,
			x=0.3cm, y=0.3cm, %sets the axis units, effectively controlling size of picture
			xlabel={$x_1$},
			ylabel={$x_2$},
			xlabel style={right},
			ylabel style={above},
			xtick={-4,4,8,12},
			ytick={-4,4,8,12,16,20,24},
			minor xtick={-6,-4,...,8,10,12,14}, %for grid
			minor ytick={-6,-4,...,26}, %for grid
			grid=minor,
			legend cell align={left}, 
			xmin=-6.9,
			xmax=14.9,
			ymin=-6.9,
			ymax=26.9]
				
	  \filldraw (-2,6) circle (1.5pt) node[above] {$\bxh$}; % observed solution		
		%\draw[->] (0,0) -- (-2,6) node[above] {$\bxh$}; % observed solution		
		\draw[semithick] (-6,22) -- (8,-6) -- (-6,-6) -- (-6,22); % nominal feasible region
		\draw[very thick,red,dashed] (-4,14) -- (0,10) -- (10/3,0) -- (14/3,-4) -- (-4,-4) -- (-4,14); % robust feasible region
		\draw[very thick,blue,dotted] (-4,7) -- (0,5) -- (10/3,0) -- (10/3,-4); % moved constraint
		\draw[->, very thick, blue,dotted] (-1,5.5) -- (-3,1.5) node[below] {$\bc$}; % cost vector
		
		% Legend
		\addlegendimage{semithick}
		\addlegendentry{Nominal feasible region}
		\addlegendimage{very thick,red,dashed}
		\addlegendentry{Prior robust feasible region}
		\addlegendimage{very thick,blue,dotted}
		\addlegendentry{Imputed robust constraint $i^*$}
		
		\end{axis}
\end{tikzpicture}
\caption{Example \ref{ex:IuSd}: RLO-IU-SD.}
\label{fig:IuSd}
\end{subfigure}
\caption{Numerical examples for the interval uncertainty IO models. Both examples share the same observed solution and nominal feasible region.}
\label{fig:IuEx}
\end{figure*}

\begin{exmp}[\textbf{RLO-IU-DG}]\label{ex:IuDg}
Consider the nominal problem
\begin{align*}%\label{rlpex}
\minimize_{\bx} \;\;\; & c_1 x_1 + c_2 x_2 \\
\mbox{subject to} \;\;\; & x_1 \geq -6, \\
& x_2 \geq -6, \\
& -2 x_1 - x_2 \geq -10.
\end{align*}
Let the constraints and variables be indexed by $I = \{1,2,3\}$ and $J = \{1,2\}$ respectively, and let the coefficients subject to uncertainty be defined by $J_1 = \{1\}, J_2 = \{2\}, J_3  = \{1,2\}$.

Let $\bxh = (-2,6)$ be the observed solution, and let the side constraints be defined by
\begin{align*}
\bOmega = \left\{\balpha \colon \alpha_{ij} \geq 0.5, \forall j\in J_i, i\in I ; \sum_{i\in I}\sum_{j\in J_i}\alpha_{ij} \le 2.5 \right\}.
\end{align*}
By Proposition \ref{pro:IuTwoFeas}, RLO-IU-DG is feasible if and only if $\bxh$ is robust feasible with respect to some $\balpha \in \bOmega$; in this example, this requirement can be met by $\alpha_{ij} = 0.5$ for all $j\in J_i, i\in I$, so the IO problem is feasible.

The nominal and imputed robust feasible regions are shown in Figure~\ref{fig:IuDg}. Applying Theorem \ref{thm:IuDg}, we find that $\bt = (2, 6, 1)$, so $i^* = 3$. An optimal solution of \eqref{IuMipt} corresponding to $i^*$ has $\alpha^*_{11} = 0.5$, $\alpha^*_{22} = 0.5$, and $\balpha^*_3 = (0.5,1)$. The robust counterpart of the third constraint is equivalent to $-2 x_1 - x_2 - 0.5|x_1| - |x_2| \geq -10$, thus the realization of the constraint is different in each quadrant. All three constraints have a positive surplus, meaning that the observed solution could not be made exactly optimal. The minimum duality gap is obtained by the cost vector $\bc = (-1.5,-2)$, which is perpendicular to the third constraint in the same quadrant as $\bxh$.
\end{exmp}

\begin{exmp}[\textbf{RLO-IU-SD}]\label{ex:IuSd}
Let the observed solution, nominal problem, and index sets be the same as in Example \ref{ex:IuDg}. Let the robust optimization problem have the given prior parameters $\hat{\alpha}_{11} = 0.5, \hat{\alpha}_{22} = 0.5, \balphah_{3} = (1,0)$. The nominal and robust (assuming $\hat\balpha$) feasible regions are shown in Figure~\ref{fig:IuSd}; in particular, the robust counterpart of the third constraint is equivalent to $-2 x_1 - x_2 - |x_1| \geq -10$.

Given this forward problem, the corresponding IO problem RLO-IU-SD is feasible, since $\bxh$ is feasible for the nominal problem (see Proposition \ref{pro:IuFeas}). For simplicity, we use the $L_1$ norm and the weight vector $\bxi = \beee$ in the objective function. Applying Theorem \ref{thm:IuAlg}, we find $\bt = (1.5, 1.5, 1)$, so constraint $i^* = 3$ will be set active. An optimal solution of formulation \eqref{Iuin1eq} for $i^*$ has $\balpha^*_1 = \balphah_1$ and $\balpha^*_2 = \balphah_2$ (since $\bxh$ is feasible for the prior $\balphah$), which were the same values imputed in Example~\ref{ex:IuDg}. For the third constraint, we find $\balpha^*_3 = (1, 1)$, which is a ``larger'' uncertainty set (i.e., smaller feasible region) than the one imputed in Example~\ref{ex:IuDg}. The robust counterpart of the third constraint then becomes $-2 x_1 - x_2 - |x_1| - |x_2| \geq -10$, and $\bxh$ satisfies this constraint with equality in the second quadrant. Accordingly, the imputed cost vector $\bc = \bc^3 = (-1,-2)$ is perpendicular to the third constraint in the same quadrant as $\bxh$.
\end{exmp}

\section{Robust linear optimization with cardinality constrained uncertainty}\label{sec:ccu}

In this section, we consider a robust linear optimization problem with a cardinality constrained uncertainty set \citep{bertsimas2004}, assuming a nearly identical setup as in the previous section. For each constraint $i \in I$, this uncertainty set bounds the number of uncertain coefficients $\tilde a_{ij}$ that can deviate from their nominal value $a_{ij}$ within the range $[a_{ij} - \alpha_{ij}, a_{ij} + \alpha_{ij}]$, for all $j \in J_i$, using a budget parameter $\Gamma_i$. The robust problem is:
\begin{subequations}\label{Ccuzero}
\begin{align}
\minimize_{\bx} \;\;\; & \sum_{j\in J}c_j x_j \\
\mbox{subject to} \;\;\; & \sum_{j \in J} a_{ij} x_{j} - \max_{\substack{\{S_i\cup\{r_i\}\colon S_i\subseteq J_i, \\ |S_i|=\lfloor\Gamma_i\rfloor, r_i\in J_i\setminus S_i \}}}
\left\{\sum_{j\in S_i}\alpha_{ij} \lvert x_{j}\rvert + (\Gamma_i - \lfloor \Gamma_i\rfloor)\alpha_{i r_i} \lvert x_{r_i}\rvert \right\} \geq b_{i}, \quad \forall i\in I. \label{Ccu1Constr}
\end{align}
\end{subequations}
For convenience, we refer to the embedded maximization problem in constraint~\eqref{Ccu1Constr} as the \emph{protection function}. When $\Gamma_i = |J_i|$, the protection function equals $\sum_{j\in J_i} \alpha_{ij}|\hat{x}_j|$ and \eqref{Ccu1Constr} becomes equivalent to the corresponding constraint of the robust linear program with interval uncertainty. Constraint \eqref{Ccu1Constr} can be linearized to yield the equivalent robust counterpart \citep{bertsimas2004}:
\begin{subequations}\label{Ccu}
\begin{align}
\minimize_{\bx,\by,\bz,\bu} \;\;\; & \sum_{j \in J} c_{j} x_{j} \\
\mbox{subject to} \;\;\; & \alpha_{ij} x_{j} + u_{ij} \geq 0, \quad \forall j\in J_i, i\in I, \label{Ccua} \\
& - \alpha_{ij} x_{j} + u_{ij} \geq 0, \quad \forall j\in J_i, i\in I, \\
& y_{ij} + z_i -u_{ij} \geq 0, \quad \forall j\in J_i, i\in I, \\
& \sum_{j \in J} a_{ij} x_{j} -\sum_{j\in J_i}y_{ij} - \Gamma_i z_i \geq b_{i}, \quad \forall i\in I, \label{Ccud} \\
& y_{ij}, z_i \geq 0, \quad \forall j\in J_i, i\in I.
\end{align}
\end{subequations}
Alternatively, if we let $j^i_k(\bx)$ index the $k$-th largest element in the set $\{\alpha_{ij}\lvert x_j \rvert\}_{j\in J_i}$, for all $k=1,\dots,|J_i|, i\in I$, then each constraint $i\in I$ can be written as $\sum_{j\in J} \bar{a}_{ij}(\Gamma_i, \bx) x_j \geq b_i$, where
\begin{align}\label{eq:abar_ccu}
\bar{a}_{ij}(\Gamma_i, \bx) & = \left\{
\begin{array}{ll}
a_{ij} - \sgn({x_j})\alpha_{ij} & \text{if } j = j^i_k(\bx), k = 1,\dots,\lfloor\Gamma_i\rfloor, \\
a_{ij} - \sgn({x_j})\alpha_{ij}(\Gamma_i - \lfloor\Gamma_i\rfloor) & \text{if } j = j^i_{\lfloor\Gamma_i\rfloor + 1}(\bx), \\
a_{ij} & \text{otherwise},
\end{array}
\right.
\end{align}
for all $\Gamma_i\in [0, |J_i|], i\in I, \bx\in\mathbb{R}^n$. In the development below, we will use the tractable robust counterpart \eqref{Ccu} as our forward problem, however, we will also use the representation \eqref{eq:abar_ccu} of the left-hand-side coefficients of a constraint where convenient.

Given $\ba_i, b_i, J_i$ and $\balpha_i$ for all $i \in I$, and a feasible $\hat\bx$ for the nominal problem, our IO problem aims to determine parameters $\Gamma_i \in [0, |J_i|]$ for all $i \in I$ such that $\hat\bx$ is minimally suboptimal for some nonzero cost vector. Note the slight difference from the interval uncertainty case: here, $\balpha_i$ is fixed as opposed to variable, and the new parameter $\Gamma_i$ is the primary variable in the inverse problem that determines the uncertainty set. As in previous sections, we propose two IO models: the first minimizes the duality gap, while the second requires the optimality conditions to be satisfied exactly. As in the case of interval uncertainty, the first model identifies uncertainty set parameters such that the surplus for a single constraint of the robust problem is minimized, while the second model identifies uncertainty set parameters such that some constraint is rendered active.

In the cardinality constrained uncertainty case, the nominal surplus of each constraint of the forward problem allows us to draw two conclusions about whether the inverse problem will be feasible, and for what values of the parameters $\Gamma_i$. First, we will show that nominal feasibility of $\bxh$ will be a necessary condition for feasibility of both IO models, and accordingly we make the following assumption:
\begin{assumption}\label{assumption:nomfeas}
$\sum_{j \in J} a_{ij} \hat{x}_{j} \geq b_{i}$ for all $i\in I$.
\end{assumption}
\noindent
This assumption was also implicitly necessary in the interval uncertainty case, and we only formalize it here because it will be invoked in multiple results below. As before, the assumption is needed because the robust feasible region must be a subset of the nominal feasible region, so a nominally infeasible solution cannot be rendered robust feasible by any choice of uncertainty set parameters. Second, if the nominal surplus for a constraint does not exceed the maximum value of the protection function, then the upper bound on $\Gamma_i$ such that the constraint is satisfied will be less than or equal to $|J_i|$. To identify these constraints, we define the set $\hat{I} := \{i\in I \colon 0 \le \sum_{j\in J} a_{ij}\hat{x}_j - b_i \le \sum_{j\in J_i} \alpha_{ij}|\hat{x}_j| \} \subseteq I$.

For $i\in\hat{I}$, we need to determine the maximum value of $\Gamma_i$ such that the constraint is satisfied. For all $i\in\hat{I}$, let $\Gamma_i = \underline{\Gamma}_i$ satisfy
\begin{align}
\sum_{j\in J}a_{ij}\hat{x}_j - b_i =& \sum_{k=1}^{\lfloor\Gamma_i\rfloor} \alpha_{i{j^i_k}(\bxh)}\lvert\hat{x}_{j^i_k(\bxh)} \rvert + (\Gamma_i - \lfloor\Gamma_i\rfloor)  \alpha_{i{j^i_{\lceil\Gamma_i\rceil}}(\bxh)}\lvert\hat{x}_{j^i_{\lceil\Gamma_i\rceil}(\bxh)} \rvert. \label{gammabar}
\end{align}
In other words, $\underline{\Gamma}_i\in [0,|J_i|]$ is a budget parameter such that the nominal surplus of constraint $i$ equals the value of the protection function, thereby rendering constraint $i$ active at $\hat\bx$. For each $i \in \hat{I}$, $\underline{\Gamma}_i$ can be computed as the optimal value of the following linear optimization problem:
\begin{align}\label{gammabaraux}
\underline{\Gamma}_i := \min_{\bzero \le \bw \le \beee} \left\{\sum_{j\in J_i}w_j \colon \sum_{j\in J_i} \alpha_{ij}\lvert\hat{x}_j\rvert w_j = \sum_{j\in J}a_{ij}\hat{x}_j - b_i \right\}.
\end{align}

To simplify the presentation, we will make the following assumption, which is without loss of generality.

\begin{assumption}\label{assumption:gammabar}
$\underline{\Gamma}_i$ is the unique solution to equation \eqref{gammabar}.
\end{assumption}
\noindent
Under this assumption, constraint $i$ will be infeasible for $\Gamma_i > \underline{\Gamma}_i$, and will have positive surplus for $\Gamma_i < \underline{\Gamma}_i$. This assumption is without loss of generality because equation \eqref{gammabar} would otherwise be satisfied by any $\Gamma_i \in [\underline{\Gamma}_i,\overline{\Gamma}_i]$, where
\begin{align*}
\overline{\Gamma}_i =
  \begin{cases}
   \lvert J_i \rvert & \text{if } \sum_{j\in J} a_{ij}\hat{x}_j - b_i = \sum_{j\in J_i} \alpha_{ij}|\hat{x}_j|, \\
   \underline{\Gamma}_i & \text{otherwise},
  \end{cases}
\end{align*}
and correspondingly, constraint $i$ would be infeasible for $\Gamma_i > \overline{\Gamma}_i$, and would have positive surplus for $\Gamma_i < \underline{\Gamma}_i$. The results in the remainder of this section would change by simply requiring $\Gamma_i\in[0,\overline{\Gamma}_i]$ wherever we currently have $\Gamma_i\in[0,\underline{\Gamma}_i]$, and $\Gamma_i\in[\underline{\Gamma}_i,\overline{\Gamma}_i]$ wherever we currently have $\Gamma_i = \underline{\Gamma}_i$. Because the right-hand side of equation~\eqref{gammabar} is strictly increasing in $\Gamma_i$ if $\alpha_{ij}|\hat{x}_j| > 0$ for all $j\in J_i$, it can easily be shown that multiple possible $\Gamma_i$ will satisfy equation~\eqref{gammabar} if there are one or more indices $j \in J_i$ such that $\alpha_{ij}|\hat{x}_j| = 0$ and $\sum_{j\in J}a_{ij}\hat{x}_j - b_i = \sum_{j\in J_i}\alpha_{ij}|\hat{x}_j|$.

Finally, we will make an additional assumption to prevent the IO models from returning a trivial solution, i.e., a solution in which either $\bc = \bzero$ or $\bab_i(\Gamma_i, \bx) = \bzero$ for some $i\in I, \bx\in\mathbb{R}^n$:
\begin{assumption}\label{assumption:ccu_trivial}
For all $i\in I$, $|a_{ij}| > \alpha_{ij}$ for some $j\in J_i$, or $a_{ij} \ne 0$ for some $j\in J\setminus J_i$.
\end{assumption}
\noindent
This assumption requires that each constraint of the forward problem has at least one left-hand-side coefficient for which the magnitude of the nominal value is greater than the maximum possible deviation from the nominal value, or which is known with certainty to be nonzero (as in Assumption \ref{assumption:iu_trivial}).

\subsection{Inverse optimization models}\label{sec:CcuDg}

Let $\lambda_{ij}, \mu_{ij}, \varphi_{ij}, \pi_i$ be the dual variables corresponding to constraints \eqref{Ccua}-\eqref{Ccud}, respectively. The following formulation minimizes the duality gap while enforcing convex side constraints $\bGamma \in \bOmega$, and primal and dual feasibility:
\begin{subequations}\label{Ccuin2}
\begin{align}
\textrm{RLO-CCU-DG:}\quad \minimize_{\substack{\bGamma,\bc,\bu,\by,\bz,\\\bpi,\bvarphi,\blambda,\bmu}} \;\;\; & \sum_{j \in J} c_{j} \hat{x}_{j} - \sum_{i\in I} b_i \pi_i \\ %\label{Ccuin2dualitygap}
\mbox{subject to} \;\;\; & \bGamma \in \bOmega, \label{Ccuin2side} \\
& \alpha_{ij} \hat{x}_{j} + u_{ij} \geq 0, \quad \forall j\in J_i, i\in I, \label{Ccuin1p1} \\
& -\alpha_{ij} \hat{x}_{j} + u_{ij} \geq 0, \quad \forall j\in J_i, i\in I, \label{Ccuin1p2} \\
& y_{ij} + z_i \geq u_{ij}, \quad \forall j\in J_i, i\in I, \label{Ccuin1p3} \\
& \sum_{j \in J} a_{ij} \hat{x}_{j} - \sum_{j\in J_i}y_{ij} - \Gamma_i z_i \geq b_{i}, \quad \forall i\in I, \label{Ccuin1p4} \\
& y_{ij}, z_i \geq 0, \quad \forall j\in J_i, i\in I, \label{Ccuin1p5} \\
& 0\le \Gamma_i \le |J_i|, \quad \forall i\in I, \label{Ccuin1u} \\
& \sum_{i\in I}\pi_i = 1, \label{Ccuin1norm} \\
& \sum_{i\in I}a_{ij} \pi_i + \sum_{i\in I\colon j\in J_i}\alpha_{ij} (\lambda_{ij} - {\mu}_{ij}) = c_j, \quad \forall j\in J, \label{Ccuin1d1} \\
& \varphi_{ij} \le \pi_i, \quad \forall j\in J_i, i\in I, \label{Ccuin1d2} \\
& \varphi_{ij} = {\lambda}_{ij} + {\mu}_{ij}, \quad \forall j\in J_i, i\in I, \label{Ccuin1d3} \\
& \sum_{j\in J_i}\varphi_{ij} \le \Gamma_i \pi_i, \quad \forall i\in I, \label{Ccuin1d4} \\
& {\pi}_i, {\varphi}_{ij}, {\lambda}_{ij}, {\mu}_{ij} \geq 0, \quad \forall j\in J_i, i\in I. \label{Ccuin1d5}
\end{align}
\end{subequations}
The construction of RLO-CCU-DG parallels that of RLO-IU-DG. Constraints \eqref{Ccuin1p1}-\eqref{Ccuin1p5} and \eqref{Ccuin1d1}-\eqref{Ccuin1d5} represent primal feasibility and dual feasibility, respectively. We use the same normalization constraint \eqref{Ccuin1norm} to prevent the trivial solution $(\bc,\bpi) = (\bzero,\bzero)$ from being feasible.

While there are similarities, there are also important differences between RLO-CCU-DG and RLO-IU-DG. Notably, while both formulations have bilinear constraints, their structure is different and therefore different analysis and solution methods are required. In RLO-IU-DG, the dual feasibility constraint \eqref{Iuin1d1} was bilinear in $\balpha$ and $(\blambda, \bmu)$, and although an identical constraint appears in RLO-CCU-DG, it is linear because the parameters $\balpha$ are assumed to be known. Instead, the primal feasibility constraint \eqref{Ccuin1p4} is bilinear in $\bGamma$ and $\bz$, and the dual feasibility constraint \eqref{Ccuin1d4} is bilinear in $\bGamma$ and $\bpi$.

First, we present a result that enables us to tractably deal with the bilinearity in \eqref{Ccuin1p4}, and characterize the feasibility of RLO-CCU-DG. For convenience, we define
\begin{align*}
\bTheta = \{\bGamma : \Gamma_i \in [0,\underline\Gamma_i], i \in \hat{I}; \Gamma_i \in [0,\lvert J_i\rvert], i \in I\setminus\hat{I} \},
\end{align*}
which will be used in several results below. The set $\bTheta$ defines the allowable values of $\bGamma$ such that $\bxh$ is feasible for the forward problem, or equivalently that the value of the protection function does not exceed the nominal surplus for each constraint of the forward problem. The following lemma formalizes this:
\begin{lem}\label{lem:CcuPfEq}
If $(\bGamma,\bu,\by,\bz)$ satisfies constraints \eqref{Ccuin1p1}-\eqref{Ccuin1u}, then $\bGamma \in \bTheta$. Conversely, given Assumption \ref{assumption:nomfeas}, if $\bGamma \in \bTheta$ then there exists $(\bu,\by,\bz)$ such that $(\bGamma,\bu,\by,\bz)$ satisfies constraints \eqref{Ccuin1p1}-\eqref{Ccuin1u}.
\end{lem}
\noindent
By showing that the bounds $\bGamma \in \bTheta$ are both necessary and sufficient for primal feasibility to be satisfied, Lemma \ref{lem:CcuPfEq} allows us to eliminate the bilinear constraint \eqref{Ccuin1p4} as well as the auxiliary variables $\bu,\by,\bz$ from RLO-CCU-DG.

Analogous to RLO-IU-DG, the feasibility of $\bxh$ for the nominal problem, along with extra conditions on $\bGamma$, are necessary and sufficient conditions for feasibility of RLO-CCU-DG.
\begin{pro}
\label{pro:CcuTwoFeas}
RLO-CCU-DG is feasible if and only if Assumption \ref{assumption:nomfeas} holds and $\bTheta \cap \bOmega \ne \varnothing$.
\end{pro}
\noindent
As discussed earlier, Assumption \ref{assumption:nomfeas} ($\sum_{j \in J} a_{ij} \hat{x}_{j} \geq b_{i}, \forall i\in I$) and $\bGamma\in\bTheta$ are both required for primal feasibility of the robust problem \eqref{Ccu} to be satisfied. Proposition \ref{pro:CcuTwoFeas} shows that feasibility of the IO problem requires that the side constraints $\bOmega$ allow $\bGamma$ to take a value in $\bTheta$.

Like the previous duality gap minimization models, a solution method for RLO-CCU-DG can be developed that involves solving $m$ convex optimization problems. As in the previous cases, the derivation of a solution method involves substituting the dual feasibility constraint with $\bc$ into the objective function, and then showing through algebraic analysis that there exists an optimal solution with binary $\bpi$. Due to the different structure of the bilinearities in RLO-CCU-DG, the intermediate algebraic analysis involves reasoning we did not use previously, but reaches the same conclusion. The interpretation of the following theorem is conceptually similar to Theorems \ref{thm:LpMip} and \ref{thm:IuDg}.
\begin{thm}
\label{thm:CcuDg}
For all $i\in I$, let $t_i$ be the optimal value and let $(\bGamma^{(i)}, \bvarphi^{(i)})$ be an optimal solution for the problem
\begin{equation}
\begin{aligned}\label{CcuMipt}
\minimize_{\bGamma, \bvarphi} \;\;\; & \sum_{j\in J} a_{ij} \hat{x}_j - \sum_{j\in J_i}\alpha_{ij}|\hat{x}_j|\varphi_{ij} - b_i \\
\mbox{\emph{subject to}} \;\;\; & \sum_{j\in J_i} \varphi_{ij} \le \Gamma_i, \\
& 0 \le \varphi_{ij} \le 1, \quad \forall j\in J_i, \\
& \bGamma \in \bTheta\cap\bOmega.
\end{aligned}
\end{equation}
Let $i^* \in \argmin_{i\in I} \{t_i\}$, and let $\bGamma^* = \bGamma^{(i^*)}$. Then the optimal value of RLO-CCU-DG is $t_{i^*}$, and an optimal solution $(\bGamma,\bc,\bpi)$ is
\begin{align}
\bGamma &= \bGamma^*, \label{CcuMipopt1} \\
\bc &= \bab_{i^*}(\Gamma^*_{i^*}, \bxh), \label{CcuMipopt2} \\
\bpi &= \beee_{i^*},
\end{align}
where, given Assumption \ref{assumption:ccu_trivial}, $\bc \ne \bzero$ and $\bab_i(\Gamma_i,\bx) \ne \bzero$ for all $i\in I, \bx\in\mathbb{R}^n$.
\end{thm}

\begin{rmk}\label{rmk:CcuDg}
Theorem \ref{thm:CcuDg} shows that an optimal solution to the nonconvex inverse problem RLO-CCU-DG can be found by solving at most $m$ linear optimization problems of the form \eqref{gammabaraux}, to determine the parameters $\mathbf{\underline{\Gamma}}$ defining $\bTheta$, and $m$ convex optimization problems \eqref{CcuMipt} which become linear whenever the constraints $\bGamma \in \bOmega$ can be written linearly.
\end{rmk}

\subsubsection{Enforcing strong duality}\label{sec:CcuSd}

Next, we formulate an IO model that minimizes the deviation of $\bGamma$ from given values $\bGammah$, while enforcing strong duality, and primal and dual feasibility.
\begin{subequations}\label{Ccuin1}
\begin{align}
\textrm{RLO-CCU-SD:}\quad \minimize_{\substack{\bGamma,\bc,\bu,\by,\bz,\\\bpi,\bvarphi,\blambda,\bmu}} \;\;\; & \lVert\bGamma - \bGammah \rVert \label{Ccuin1obj} \\
\mbox{subject to} \;\;\; & \sum_{j \in J} c_{j} \hat{x}_{j} - \sum_{i\in I} b_i \pi_i = 0, \label{Ccuin1sd} \\
& \eqref{Ccuin1p1}-\eqref{Ccuin1d5}.
\end{align}
\end{subequations}
In the objective function \eqref{Ccuin1obj}, $\lVert \cdot \rVert$ is an arbitrary norm. Constraint \eqref{Ccuin1sd} represents strong duality. As in RLO-CCU-DG, all constraints are linear except for the bilinear primal feasibility constraint \eqref{Ccuin1p4} and dual feasibility constraint \eqref{Ccuin1d4}.

First, we characterize the feasibility of RLO-CCU-SD.
\begin{pro}
\label{pro:CcuFeas}
RLO-CCU-SD is feasible if and only if Assumption \ref{assumption:nomfeas} holds and $\hat{I}\ne\varnothing$.
\end{pro}
\noindent
Whereas Assumption \ref{assumption:nomfeas} was both necessary and sufficient for feasibility of RLO-IU-SD, it is insufficient in the case of RLO-CCU-SD: an additional condition ($\hat{I}\ne\varnothing$) is required for strong duality to be satisfied. Therefore, the absence of side constraints on the uncertainty set parameters does not automatically mean that RLO-CCU-SD can be used rather than RLO-CCU-DG.

Whereas Lemma~\ref{lem:CcuPfEq} in the previous section showed that the primal feasibility constraints can be replaced by the bounds $\bGamma\in\bTheta$, the following lemma uses Lemma~\ref{lem:CcuPfEq} to make a stronger and more general statement that additionally addresses the strong duality and dual feasibility conditions. In particular, to satisfy strong duality, $\bGamma$ must be chosen such that $\bxh$ lies on the boundary of the robust feasible region, which corresponds to a choice of $\bGamma$ such that for at least one constraint, the protection function equals the nominal surplus.

\begin{lem}\label{lem:CcuFeas2}
Every feasible solution for RLO-CCU-SD satisfies $\bGamma \in \bTheta$ with $\Gamma_{\hat{i}} = \underline{\Gamma}_{\hat{i}}$ for a specific $\hat{i} \in \hat{I}$. Conversely, given Assumption \ref{assumption:nomfeas}, for every $\bGamma \in \bTheta$ satisfying $\Gamma_{\hat{i}} = \underline{\Gamma}_{\hat{i}}$ for a specific $\hat{i} \in \hat{I}$, there exists $(\bc,\bu,\by,\bz,\bpi,\bvarphi,\blambda,\bmu)$ such that $(\bGamma,\bc,\bu,\by,\bz,\bpi,\bvarphi,\blambda,\bmu)$ is feasible for RLO-CCU-SD.
\end{lem}

Lemma \ref{lem:CcuFeas2} will allow us to easily circumvent the bilinearity in constraint \eqref{Ccuin1d4} and characterize an optimal solution to RLO-CCU-SD. We first make the following assumption on $\bGammah$.
\begin{assumption}\label{assumption:gammabounds}
$\hat{\Gamma}_i \in [0, \lvert J_i\rvert]$ for all $i\in I$.
\end{assumption}
\noindent
This assumption is without loss of generality because any $\hat\Gamma_i$ outside the interval can be moved to the closest end point of the interval without changing the solution to RLO-CCU-SD. We now describe an efficient solution method for RLO-CCU-SD, similar to Theorem \ref{thm:LpAlg}:

\begin{thm}\label{thm:CcuAlg}
Let
\begin{align}
f_i & = \label{Ccu2act}
	\begin{cases}
		\underline{\Gamma}_i - \hat{\Gamma}_i & \text{if } i\in \hat{I}, \\
		0 & \text{if } i\in I\setminus\hat{I},
	\end{cases} \\
g_i & = \min \left\{ f_i, 0 \right\}, \quad \forall i\in I, \label{Ccu2feas} \\
c^i_j & = \bar{a}_{ij}(\Gamma_i, \bxh), \quad \forall j\in J, i\in \hat{I}, \label{Ccu2cmin} \\
i^* & \in \argmin_{i\in \hat{I}} \{ \lVert\bg + (f_i - g_i)\beee_i\rVert \}. \label{Ccu2isol}
\end{align}
Given Assumption \ref{assumption:gammabounds}, the optimal value of RLO-CCU-SD is $\lVert\bg + (f_{i^*} - g_{i^*})\beee_{i^*}\rVert$, and there exists an optimal solution with
\begin{align}
& \Gamma_i =
  \begin{cases}
   \underline{\Gamma}_i & \text{if } i=i^*, \\
   \min\{\hat{\Gamma}_i, \underline{\Gamma}_i\} & \text{if } i\in \hat{I}\setminus\{i^*\}, \\
	 \hat{\Gamma}_i & \text{if } i\in I\setminus\hat{I},
  \end{cases} \label{CcuTwogsol} \\
& \bc = \bc^{i^*}, \label{CcuTwocsol}
\end{align}
where, given Assumption \ref{assumption:ccu_trivial}, $\bc \ne \bzero$ and $\bab_i(\Gamma_i, \bx) \ne \bzero$ for all $i\in I, \bx\in\mathbb{R}^n$.
\end{thm}

\begin{rmk}
Theorem \ref{thm:CcuAlg} shows that an optimal solution to the nonconvex inverse problem RLO-CCU-SD can be found by solving at most $m$ linear optimization problems of the form \eqref{gammabaraux} to determine $\underline{\Gamma}_i$ for all $i\in I$.
\end{rmk}

While the interpretation of Theorem \ref{thm:CcuAlg} is conceptually similar to the interpretation of Theorem \ref{thm:LpAlg}, there are some differences. First, $f_i$ and $g_i$ here correspond to the values of the $i$-th component of the vector inside the norm in the objective function of the IO problem, whereas in the former result they correspond to the values of the $i$-th term in the objective function. Second, $f_i$ and $g_i$ here are functions of $\underline{\Gamma}_i$, which is the optimal value of a linear optimization problem, whereas $f_i$ and $g_i$ in Theorem \ref{thm:LpAlg} are the values of closed form expressions. Although the solution methods of Theorems \ref{thm:CcuAlg} and \ref{thm:IuAlg} appear somewhat different, they are also conceptually similar insofar as both set a single constraint of the forward problem to be active, and set the cost vector perpendicular to the part of that constraint which is in the same orthant as the observed solution.

\subsection{Numerical examples}\label{sec:CcuEx}

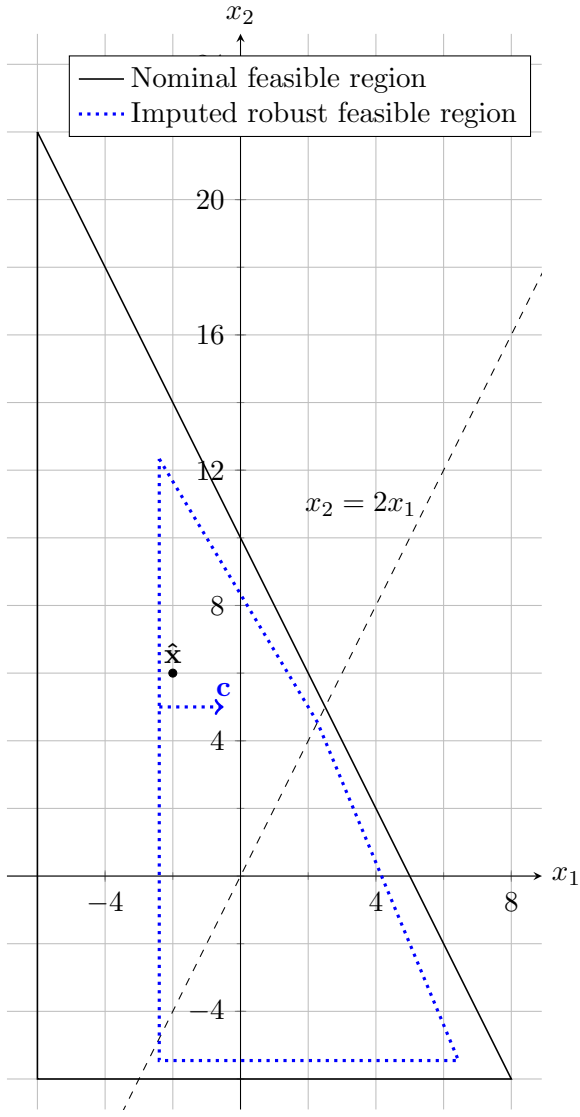
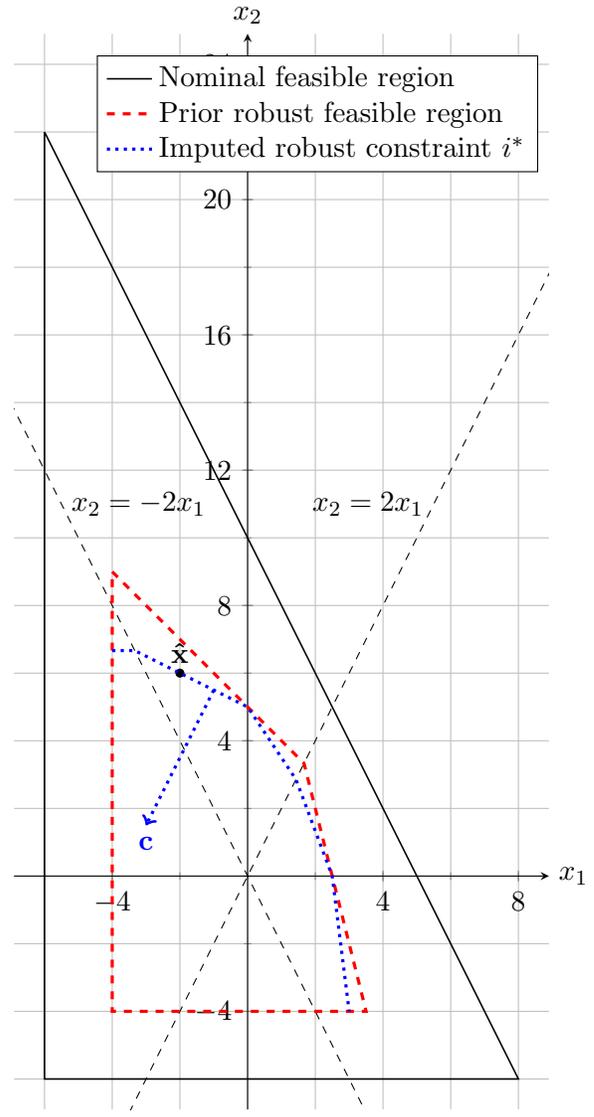
\begin{figure*}
\centering
\begin{subfigure}{0.5\textwidth}
\centering
\begin{tikzpicture}
		\begin{axis}[
			axis x line=center,
			axis y line=center,
			x=0.45cm, y=0.45cm, %sets the axis units, effectively controlling size of picture
			xlabel={$x_1$},
			ylabel={$x_2$},
			xlabel style={right},
			ylabel style={above},
			xtick={-4,4,8},
			ytick={-4,4,8,12,16,20,24},
			minor xtick={-6,-4,...,8}, %for grid
			minor ytick={-6,-4,...,24}, %for grid
			grid=minor,
			legend cell align={left}, 
			xmin=-6.9,
			xmax=8.9,
			ymin=-6.9,
			ymax=24.9]
				
		\draw[dashed] (-3.5,-7) -- (5.5,11) node[left] {$x_2 = 2x_1$} -- (9, 18); % first breakline defined by alpha
		%\draw[dashed] (-7,14) -- (-5.5,11) node[right] {$x_2 = -2x_1$} -- (3.5,-7); % second breakline defined by alpha
		\filldraw (-2,6) circle (1.5pt) node[above] {$\bxh$}; % observed solution		
		%\draw[->] (0,0) -- (-2,6) node[above] {$\bxh$}; % observed solution
		\draw[semithick] (-6,22) -- (8,-6) -- (-6,-6) -- (-6,22); % nominal feasible region
		\draw[very thick,blue,dotted] (-2.4,37/3) -- (0,50/6) -- (25/11,50/11) -- (25/6,0) -- (425/66,-60/11) -- (-2.4,-60/11) -- (-2.4,37/3); % robust feasible region
		\draw[->, very thick, blue,dotted] (-2.4,5) -- (-0.5,5) node[above] {$\bc$}; % cost vector
		
		% Legend
		\addlegendimage{semithick}
		\addlegendentry{Nominal feasible region}
		\addlegendimage{very thick,blue,dotted}
		\addlegendentry{Imputed robust feasible region}
		
		\end{axis}
\end{tikzpicture}
\caption{Example \ref{ex:CcuDg}: RLO-CCU-DG.}
\label{fig:CcuDg}
\end{subfigure}~~~~ 
\begin{subfigure}{0.5\textwidth}
\centering
\begin{tikzpicture}
		\begin{axis}[
			axis x line=center,
			axis y line=center,
			x=0.45cm, y=0.45cm, %sets the axis units, effectively controlling size of picture
			xlabel={$x_1$},
			ylabel={$x_2$},
			xlabel style={right},
			ylabel style={above},
			xtick={-4,4,8},
			ytick={-4,4,8,12,16,20,24},
			minor xtick={-6,-4,...,8}, %for grid
			minor ytick={-6,-4,...,24}, %for grid
			grid=minor,
			legend cell align={left}, 
			xmin=-6.9,
			xmax=8.9,
			ymin=-6.9,
			ymax=24.9]
				
		\draw[dashed] (-3.5,-7) -- (5.5,11) node[left] {$x_2 = 2x_1$} -- (9, 18); % first breakline defined by alpha
		\draw[dashed] (-7,14) -- (-5.5,11) node[right] {$x_2 = -2x_1$} -- (3.5,-7); % second breakline defined by alpha
		\filldraw (-2,6) circle (1.5pt) node[above] {$\bxh$}; % observed solution		
		%\draw[->] (0,0) -- (-2,6) node[above] {$\bxh$}; % observed solution
		\draw[semithick] (-6,22) -- (8,-6) -- (-6,-6) -- (-6,22); % nominal feasible region
		\draw[very thick,red,dashed] (-4,9) -- (5/3,10/3) -- (7/2,-4) -- (-4,-4) -- (-4,9); % robust feasible region
		\draw[very thick,blue,dotted] (-4,20/3) -- (-10/3,20/3) -- (0,5) -- (10/7,20/7) -- (5/2,0) -- (3,-4); % moved constraint
		\draw[->, very thick, blue,dotted] (-1,5.5) -- (-3,1.5) node[below] {$\bc$}; % cost vector
		
		% Legend
		\addlegendimage{semithick}
		\addlegendentry{Nominal feasible region}
		\addlegendimage{very thick,red,dashed}
		\addlegendentry{Prior robust feasible region}
		\addlegendimage{very thick,blue,dotted}
		\addlegendentry{Imputed robust constraint $i^*$}
		
		\end{axis}
\end{tikzpicture}
\caption{Example \ref{ex:CcuSd}: RLO-CCU-SD.}
\label{fig:CcuSd}
\end{subfigure}
\caption{Numerical examples for the cardinality constrained uncertainty IO models.}
\label{fig:CcuEx}
\end{figure*}

In this section, we provide examples illustrating the solutions of RLO-CCU-DG and RLO-CCU-SD.

\begin{exmp}[\textbf{RLO-CCU-DG}]\label{ex:CcuDg}
Let the observed solution, nominal problem, and index sets be the same as in Example \ref{ex:IuDg} (RLO-IU-DG). We assume fixed parameters $\alpha_{11} = 2.5, \alpha_{22} = 0.5, \balpha_{3} = (2,1)$. We will require that the imputed $\bGamma$ be in the set
\begin{align*}
\bOmega = \left\{\bGamma \colon \Gamma_i \geq 0.2, \forall i\in I; \sum_{i\in I}\Gamma_i \le 1 \right\}.
\end{align*}
Similar to the RLO-IU-DG case, we can verify the feasibility of the IO problem by finding uncertainty set parameters that meet the condition in Proposition \ref{pro:CcuTwoFeas} (for brevity, we omit this step).

To determine the optimal solution of RLO-CCU-DG, we apply Theorem \ref{thm:CcuDg} and find $\bt = (1, 10.2, 4.4)$, so $i^* = 1$ with corresponding $\bGamma^* = (0.6, 0.2, 0.2)$ and $\bc = (2.5,0)$. In other words, the minimum possible surplus for any of the three constraints is obtained by maximizing the degree of uncertainty associated with the first constraint. The duality gap equals the surplus of the first constraint, and the cost vector is perpendicular to the first constraint. The nominal and imputed robust feasible regions are shown in Figure~\ref{fig:CcuDg}. In particular, the robust counterpart of the third constraint is equivalent to
\begin{align*}
& -2 x_1 - x_2 - \max\left\{ 0.4|x_{1}|, 0.2|x_{2}| \right\} \geq -10.
\end{align*}
Thus, the realization of this constraint depends not only on the sign of $\bx$ (as the RLO-IU-SD example), but also the position of $\bx$ relative to the lines $x_2 = 2x_1$ and $x_2 = -2x_1$. Figure~\ref{fig:CcuDg} depicts the first of these two lines; at the point that it intersects the boundary of the constraint, the constraint changes slope.
\end{exmp}

\begin{exmp}[\textbf{RLO-CCU-SD}]\label{ex:CcuSd}
Let the observed solution, nominal problem, and index sets be the same as in Example \ref{ex:IuDg} (RLO-IU-DG), and let $\alpha_{11} = 2.5, \alpha_{22} = 0.5$, and $\balpha_{3} = (2,1)$ as in Example \ref{ex:CcuDg}. We assume a prior $\bGammah = (0.2,1,1)$. We use the $L_1$ norm for the objective function. The nominal and robust feasible regions are shown in Figure~\ref{fig:CcuSd}. In particular, the robust counterpart of the third constraint is equivalent to
\begin{align*}
& -2 x_1 - x_2 - \max\left\{ 2|x_{1}|, |x_{2}| \right\} \geq -10.
\end{align*}

First, we verify that constraints 1 and 3 (but not 2) have nominal surplus less than the maximum value of the corresponding protection function, and therefore $\hat{I} = \{1,3\}$. Solving formulation \eqref{gammabaraux} results in $\underline{\Gamma}_1 = 0.8, \underline{\Gamma}_3 = 1.5$. Next, we determine that $g_i = 0$ for all $i\in \hat{I}$, meaning that $\bxh$ is feasible with respect to the prior $\hat\bGamma$ and the problem reduces to finding $i^*\in \argmin_{i\in \hat{I}} |f_i|$. Finally, we find the unique solution $i^* = 3$ and $f_3 = 0.5$. Letting $\Gamma_{3} = \underline{\Gamma}_3 = 1.5$, the robust counterpart of the third constraint becomes
\begin{align*}
& -2 x_1 - x_2 - \max\left\{ 2|x_1| + 0.5|x_2|, |x_1| + |x_2| \right\} \geq -10,
\end{align*}
which is piecewise linear (breakpoints defined by the two coordinate axes and the two equations $x_2 = 2x_1$ and $x_2 = -2x_1$). The observed $\bxh$ is on the part of this constraint in the region defined by $x_1 < 0, x_2 \geq 0, x_2 \geq -2x_1$, and the imputed cost vector $\bc = \bc^3 = (-1,-2)$ is perpendicular to the constraint at $\bxh$. Incidentally, the constraint which is set active differs from the constraint with the minimum surplus in Example \ref{ex:CcuDg}.
\end{exmp}

\section{Conclusion}\label{sec:con}

This paper demonstrates that inverse optimization to recover parameters defining the feasible set of linear and robust linear optimization problems is tractable. In particular, finding a solution to one of our models requires solving at most a linear number of convex problems, which under mild conditions can be reduced to linear problems or even closed form solutions. Our general approach to imputing left-hand-side or uncertainty set parameters leverages the intuitive geometry associated with inverse linear optimization. Despite differences in the types of parameters being imputed, the key steps in model construction and solution method are common to all forward problems. Future work could consider generalizing our IO models to include side constraints on the cost vector, since a limitation of our models is that they may imply a cost vector which is unreasonable in the context of the application domain. Another extension could allow multiple (noisy) data points, which would require reformulating our models to allow the observed solutions to be infeasible for the forward problem. Finally, we could also consider nonlinear forward problems, which would require reformulating our models using the corresponding optimality conditions. For each of these extensions, the tractability of the resulting IO models is unknown.

\section*{Acknowledgments}

This research was partially supported by the Natural Sciences and Engineering Research Council of Canada, and the Ontario Ministry of Training, Colleges and Universities.

\bibliography{references}

\newpage
\appendix
\renewcommand{\thefigure}{\arabic{figure}}
\Large \noindent \textbf{Supplementary material} \\
\normalsize

This supplementary material contains two appendices to the main body of the article. Appendix A contains proofs and lemmas, and Appendix B contains numerical examples of how to circumvent trivial solutions to NLO-SD that may be induced when Assumption 2 is violated.

\section{Proofs and supplementary lemmas}\label{sec:proofs}

\begin{proof}[Proof of Lemma \ref{lem:LpFeas2}]
To prove the first statement, we assume $\sum_{j\in J}a_{ij} \hat{x}_j > b_i$ for all $i\in I$ and derive a contradiction. Substituting \eqref{Lpin1d} into \eqref{Lpin1sd}, we get
\begin{equation}\label{nom:temp1}
\sum_{i\in I} \pi_i \sum_{j \in J}a_{ij} \hat{x}_{j} = \sum_{i\in I} \pi_i b_i.
\end{equation}
Constraint \eqref{Lpin1norm} ensures that $\overline{I} := \{i\in I \colon \pi_i > 0\} \ne \varnothing$. Since we have assumed $\sum_{j \in J} a_{ij} \hat{x}_{j} > b_{i}$ for all $i\in I$, we have
\begin{align*}
\pi_{i}\sum_{j \in J} a_{ij} \hat{x}_j > \pi_{i}b_{i}, \quad \forall i \in \overline{I},
\end{align*}
and since $\pi_i \geq 0$ for all $i\in I$,
\begin{align*}
\sum_{i\in I}\pi_{i}\sum_{j \in J} a_{ij} \hat{x}_j > \sum_{i\in I}\pi_{i}b_{i},
\end{align*}
which contradicts equation \eqref{nom:temp1}.

To prove the second statement, let $\bA$ satisfy \eqref{Lpin1abceq}, $\hat{i}$ be defined by \eqref{thm:Lp2algc1}, $\bc = \ba_{\hat{i}}$ and $\bpi = \beee_{\hat{i}}$. This solution is feasible for NLO-SD.
\end{proof}

\begin{proof}[Proof of Proposition \ref{pro:IuTwoFeas}] $(\Rightarrow)$ Assume that every nonnegative $\balpha\in \bOmega$ has some $\hat{i}\in I$ such that $\sum_{j \in J} a_{\hat{i}j} \hat{x}_{j} - \sum_{j \in J_{\hat{i}}} \alpha_{\hat{i}j} \lvert\hat{x}_{j}\rvert < b_{\hat{i}}$. Constraints \eqref{Iuin1p1}-\eqref{Iuin1p2} and \eqref{Iuin1u} imply $u_{\hat{i}j} \geq \alpha_{\hat{i}j}\hat{x}_j$ for all $j\in J_{\hat{i}}$. It follows that $\sum_{j \in J} a_{\hat{i}j} \hat{x}_{j} - \sum_{j\in J_{\hat{i}}}u_{\hat{i}j} < b_{\hat{i}}$, meaning that the constraint \eqref{Iuin1p3} is violated for $\hat{i}$.

$(\Leftarrow)$ Let $\hat{i}\in I$ be an arbitrary index, and let $\balpha$ take any nonnegative value satisfying the condition stated in the proposition. Then, it is easy to check that the following is a feasible solution for RLO-IU-DG:
\begin{equation}\label{IuFeasSol}
\begin{aligned}
\bpi &= \beee_{\hat{i}}, \\
(\lambda_{ij},\mu_{ij}) &= \left\{
\begin{array}{ll}
(1,0) & \text{if } \hat{x}_j \le 0, i=\hat{i}, \\
(0,1) & \text{if } \hat{x}_j > 0, i=\hat{i}, \\
(0,0) & \text{otherwise},
\end{array}
\right. \quad \forall j\in J_i, i\in I, \\
u_{ij} &= \alpha_{ij}\lvert\hat{x}_j\rvert, \quad \forall j\in J_i, i\in I, \\
c_j &= \bar{a}_{\hat{i}j}(\balpha_{\hat{i}}, \bxh), \quad \forall j\in J.
%\left\{
%\begin{array}{ll}
%a_{\hat{i}j} - \sgn({\hat{x}_j})\alpha_{\hat{i}j} & \text{if } j\in J_{\hat{i}}, \\
%a_{\hat{i}j} & \text{if } j\in J\setminus J_{\hat{i}}.
%\end{array}
%\right.
\end{aligned}
\end{equation}
\end{proof}

\begin{proof}[Proof of Theorem \ref{thm:IuDg}] We note first that constraints \eqref{Iuin1p1}-\eqref{Iuin1p3} can be replaced by \eqref{Iuin1eq3}, since \eqref{Iuin1p1}-\eqref{Iuin1u} imply \eqref{Iuin1eq3}, and for any $\balpha$ satisfying \eqref{Iuin1eq3} and \eqref{Iuin1u}, \eqref{Iuin1p1}-\eqref{Iuin1p3} will be satisfied by $(\balpha,\bu)$ where $u_{ij} = \alpha_{ij}|\hat{x}_j|$ for all $j\in J_i, i\in I$ (this reasoning is similar to the proof of Lemma \ref{lem:IuFeas2} below).

We eliminate $\bc$ by substituting the dual feasibility constraint \eqref{Iuin1d1} into the objective function \eqref{Iuin2dualitygap}. The resulting model has an objective function that is bilinear in variables whose corresponding feasible sets $P = \{\balpha \colon \eqref{Iuin1eq3}, \balpha \in \bOmega, \balpha \geq \bzero \}$ and $D = \{(\bpi,\blambda,\bmu) \colon \eqref{Iuin1norm},\eqref{Iuin1d2}-\eqref{Iuin1d3}\}$ are disjoint:
\begin{align}\label{Iuin2a}
\minimize_{\substack{\balpha\in P, \\ (\bpi,\blambda,\bmu)\in D}} \;\;\; & \sum_{i\in I}\sum_{j \in J}a_{ij} \pi_i \hat{x}_{j} +  \sum_{i\in I} \sum_{j\in J_i}\alpha_{ij} (\lambda_{ij} - {\mu}_{ij}) \hat{x}_{j} - \sum_{i\in I} b_i\pi_i.
\end{align}
Since $D$ is a bounded polyhedron and disjoint from $P$, an optimal solution to~\eqref{Iuin2a} exists among the vertices of $D$ \citep[Proposition 3.1]{horst2000}. The constraints $\sum_{i\in I}\pi_i = 1$ and ${\pi}_i \geq 0$ for all $i\in I$ imply that a vertex of $D$ will satisfy $\pi_{\hat{i}} = 1$ for some $\hat{i}\in I$, and $\pi_i = 0$ for all $i\in I\setminus\{\hat{i}\}$. So it suffices to consider binary $\pi_i$. Let $s_{ij} = \lambda_{ij} - \mu_{ij}$ for all $j\in J_i, i\in I$. By Lemma \ref{lem:sij} (see below), constraint \eqref{Iuin1d2} and nonnegativity of $(\blambda,\bmu)$ are equivalent to $s_{ij} \in [-\pi_i,\pi_i]$ for all $j \in J_i, i \in I$. Thus, formulation~\eqref{Iuin2a} is equivalent to
\begin{align*}
\minimize_{\balpha\in P,\bpi,\bs} \;\;\; & \sum_{i\in I}\left(\sum_{j\in J}a_{ij}\hat{x}_j - b_i \right) \pi_i + \sum_{i\in I}\sum_{j\in J_i}\alpha_{ij}\hat{x}_j s_{ij} \\
\mbox{subject to} \;\;\; & -\pi_i \le s_{ij} \le \pi_i, \quad\forall j\in J_i, i\in I, \\
& \sum_{i\in I}\pi_i = 1, \\
& \pi_i \in \{0,1\}, \quad \forall i\in I.
\end{align*}

By inspection, we see that for a given $(\balpha, \bpi)$, an optimal $\bs$ satisfies $s_{ij} = -\sgn{(\hat{x}_j)}\pi_i$. This fact allows us to eliminate $\bs$:
\begin{equation}\label{Iuin2b}
\begin{aligned}
\minimize_{\balpha\in P,\bpi} \;\;\; & \sum_{i\in I}\left(\sum_{j\in J}a_{ij}\hat{x}_j - \sum_{j\in J_i}\alpha_{ij}|\hat{x}_j| - b_i \right) \pi_i \\
\mbox{subject to} \;\;\; & \sum_{i\in I}\pi_i = 1, \\
& \pi_i \in \{0,1\}, \quad \forall i\in I.
\end{aligned}
\end{equation}
For a given $\balpha\in P$, it is clear that an optimal $\bpi$ is $\beee_{i^*}$, where $i^* \in \argmin_{i\in I} \{\sum_{j\in J}a_{ij}\hat{x}_j - \sum_{j\in J_i}\alpha_{ij}|\hat{x}_j| - b_i \}$. The optimal value and solution of problem \eqref{Iuin2b}, and the associated $\bc$, can then be determined by following a similar argument as in Theorem \ref{thm:LpMip}.

By Assumption \ref{assumption:iu_trivial}, $\bab_i(\balpha_i,\bx) = \bzero$ is infeasible for problem \eqref{IuMipt} for all $i\in I, \bx\in\mathbb{R}^n$, and therefore the optimal solution \eqref{IuSdopt1}-\eqref{IuSdopt2} satisfies $\bc \ne \bzero$ and $\bab_i(\balpha_i,\bx) \ne \bzero$ for all $i\in I, \bx\in\mathbb{R}^n$.
\end{proof}

\begin{proof}[Proof of Proposition \ref{pro:IuFeas}] $(\Rightarrow)$ Assume that $\sum_{j \in J} a_{\hat{i}j} \hat{x}_{j} < b_{\hat{i}}$ for some $\hat{i}\in I$. By following a similar argument as in the proof of Proposition \ref{pro:IuTwoFeas}, it can be shown
%Constraints \eqref{Iuin1p1} and \eqref{Iuin1p2} imply $u_{\hat{i}j} \geq 0$ for all $j\in J_{\hat{i}}$. It follows that $\sum_{j \in J} a_{\hat{i}j} \hat{x}_{j} - \sum_{j\in J_{\hat{i}}}u_{\hat{i}j} < b_{\hat{i}}$, meaning
that the constraint \eqref{Iuin1p3} is violated for $\hat{i}$.

$(\Leftarrow)$ By Assumption \ref{assumption:Iu2}, there exists $\hat{i} \in I$ and $\hat{j}\in J_{\hat{i}}$ such that $\hat x_{\hat{j}} \ne 0$. Then, it is easy to check that the following is a feasible solution for RLO-IU-SD: $(\bc,\bu,\bpi,\blambda,\bmu)$ as in \eqref{IuFeasSol}, and
\begin{align*}
\alpha_{ij} &= \left\{
\begin{array}{ll}
\frac{\sum_{k \in J} a_{ik} \hat{x}_k - b_{i}}{\lvert\hat{x}_j\rvert} & \text{if } j=\hat{j}, i=\hat{i}, \\
0 & \text{otherwise}.
\end{array}
\right.
\end{align*}
\end{proof}

\begin{proof}[Proof of Theorem \ref{thm:IuAlg}] By Lemma \ref{lem:IuFeas2} (see below), solving RLO-IU-SD is equivalent to solving the following optimization problem for all $\hat{i}\in I$, and taking the minimum over all $\lvert I\rvert$ optimal values:
\begin{align*}
\minimize_{\balpha} \;\;\; & \sum_{i\in I}\xi_i \lVert \balpha_i - \balphah_i \rVert \\
\mbox{subject to} \;\;\; & \sum_{j\in J}a_{\hat{i}j} \hat{x}_j -\sum_{j\in J_{\hat{i}}}{{\alpha}_{\hat{i}j}|\hat{x}_{j}|} = b_{\hat{i}} \\ %, \quad \text{for some } \hat{i}\in I
& \sum_{j\in J}a_{ij} \hat{x}_j -\sum_{j\in J_i}{{\alpha}_{ij}|\hat{x}_{j}|} \geq b_i, \quad \forall i\in I, \\
& \alpha_{ij} \geq 0, \quad \forall j\in J_i, i\in I.
\end{align*}
Then the optimal value and an optimal $(\bc, \balpha)$ for RLO-IU-SD can be determined by following the same argument as in Theorem \ref{thm:LpAlg}.

Next we will show that either of the two conditions in Assumption \ref{assumption:iu_trivial} implies that $\bab_i(\balpha_i, \bx) \ne \bzero$ for all $i\in I, \bx\in\mathbb{R}^n$, from which $\bc \ne \bzero$ follows since $\bc = \mathbf{\bar{a}}_{i^*}(\balpha_{i^*}, \bxh)$. First we will prove by contrapositive that $b_i > 0$ implies the required conclusion. Supposing that $\bab_i(\balpha_i, \bx) = \bzero$ for some $i\in I, \bx\in \mathbb{R}^n$, the definition of $\bab_i(\balpha_i, \bx)$ implies both that $|a_{ij}| = \alpha_{ij}$ for all $j\in J_i$, and $a_{ij} = 0$ for all $j\in J\setminus J_i$. The former statement implies $a_{ij}\hat{x}_j \le \alpha_{ij} |\hat{x}_j|$ for all $j\in J_i$, whereas the latter implies $\sum_{j\in J\setminus J_i} a_{ij}\hat{x}_j = 0$, and since $\balpha_i$ must satisfy $\sum_{j\in J} a_{ij}\hat{x}_j - \sum_{j\in J_i} \alpha_{ij}|\hat{x}_j| \geq b_i$, we can deduce $b_i \le 0$. Second, the assumption that $a_{ij} \ne 0$ for some $j\in J\setminus J_i$ implies that $\bar{a}_{ij}(\balpha_i, \bx) \ne 0$ for any $i\in I, \balpha_i \geq \bzero, \bx\in\mathbb{R}^n$.
\end{proof}

\begin{proof}[Proof of Lemma \ref{lem:CcuPfEq}]
To prove the first statement, suppose to the contrary that $\Gamma_i \in (\underline{\Gamma}_i,|J_i|]$ for some $i\in \hat{I}$, i.e., $\bGamma \notin \bTheta$ but $\Gamma_i$ does satisfy constraint \eqref{Ccuin1u}. We will show that there is no $(\bu,\by,\bz)$ that satisfies constraints \eqref{Ccuin1p1}-\eqref{Ccuin1p5}. In particular, we will show that any $(\bu,\by,\bz)$ that satisfies \eqref{Ccuin1p1}-\eqref{Ccuin1p3} and \eqref{Ccuin1p5} will never satisfy \eqref{Ccuin1p4}.

Consider the following linear optimization problem:
\begin{equation}
\begin{aligned}\label{aux}
\minimize_{\bu_i,\by_i,z_i} \;\;\; & \sum_{j\in J_i}y_{ij} + \Gamma_i z_i \\
\mbox{subject to} \;\;\; & y_{ij} + z_i \geq u_{ij}, \quad \forall j\in J_i, \\
& -u_{ij} \le \alpha_{ij} \hat{x}_{j} \le u_{ij}, \quad j\in J_i, \\
& y_{ij}, z_i \geq 0, \quad \forall j\in J_i.
\end{aligned}
\end{equation}
By Lemma \ref{lem:aux} (see below), an optimal solution to this problem is
\begin{equation}
\begin{aligned}\label{Ccu_uyz}
& u^*_{ij}=\alpha_{ij} \lvert\hat{x}_{j}\rvert, \quad \forall j\in J_i, \\
& y^*_{ij}=\max\{u^*_{ij}-z^*_i,0\}, \quad\forall j\in J_i, \\
& z_i^*=\alpha_{ij^i_{\lceil\Gamma_i\rceil}(\bxh)} \lvert\hat{x}_{j^i_{\lceil\Gamma_i\rceil}(\bxh)}\rvert,
\end{aligned}
\end{equation}
which satisfies \eqref{Ccuin1p1}-\eqref{Ccuin1p3} and \eqref{Ccuin1p5} since the constraints of \eqref{aux} are identical to \eqref{Ccuin1p1}-\eqref{Ccuin1p3} and \eqref{Ccuin1p5}. However, $(\bu^*_i,\by^*_i,z^*_i)$ does not satisfy \eqref{Ccuin1p4}, as shown below. Because $(\bu^*_i,\by^*_i,z^*_i)$ yields the smallest possible value of $\sum_{j\in J_i}y_{ij} + \Gamma_i z_i$, there cannot be any other feasible solution for \eqref{Ccuin1p1}-\eqref{Ccuin1p3} and \eqref{Ccuin1p5} which will satisfy \eqref{Ccuin1p4}.

For completeness, we substitute $(\bu^*_i,\by^*_i,z^*_i)$ into the left-hand-side of constraint \eqref{Ccuin1p4} to show that the constraint will not be satisfied:
\begin{align}
& \sum_{j \in J} a_{ij} \hat{x}_{j} - \left(\sum_{j\in J_i}\max\left\{\alpha_{ij} \lvert\hat{x}_{j}\rvert-\alpha_{ij^i_{\lceil\Gamma_i\rceil}(\bxh)} \lvert\hat{x}_{j^i_{\lceil\Gamma_i\rceil}(\bxh)}\rvert,0\right\} + \Gamma_i \alpha_{ij^i_{\lceil\Gamma_i\rceil}(\bxh)} \lvert\hat{x}_{j^i_{\lceil\Gamma_i\rceil}(\bxh)}\rvert\right) \nonumber \\
\Leftrightarrow\;\;\; & \sum_{j \in J} a_{ij} \hat{x}_{j} - \sum_{k=1}^{\lfloor\Gamma_i\rfloor}\left(\alpha_{ij^i_k(\bxh)} \lvert\hat{x}_{j^i_k(\bxh)}\rvert-\alpha_{ij^i_{\lceil\Gamma_i\rceil}(\bxh)} \lvert\hat{x}_{j^i_{\lceil\Gamma_i\rceil}(\bxh)}\rvert\right) - \Gamma_i \alpha_{ij^i_{\lceil\Gamma_i\rceil}(\bxh)} \lvert\hat{x}_{j^i_{\lceil\Gamma_i\rceil}(\bxh)}\rvert \nonumber \\
\Leftrightarrow\;\;\; & \sum_{j \in J} a_{ij} \hat{x}_{j} - \sum_{k=1}^{\lfloor\Gamma_i\rfloor}\alpha_{ij^i_k(\bxh)} \lvert\hat{x}_{j^i_k(\bxh)}\rvert - (\Gamma_i - \lfloor\Gamma_i\rfloor) \alpha_{ij^i_{\lceil\Gamma_i\rceil}(\bxh)} \lvert\hat{x}_{j^i_{\lceil\Gamma_i\rceil}(\bxh)}\rvert. \label{Ccup4eq}
\end{align}
By Assumption \ref{assumption:gammabar}, $\underline{\Gamma}_i$ is the unique value of $\Gamma_i$ such that \eqref{Ccup4eq} equals $b_i$. Since $\Gamma_i > \underline{\Gamma}_i$ by assumption, \eqref{Ccup4eq} is strictly less than $b_i$, that is, $(\bu^*_i,\by^*_i,z^*_i)$ does not satisfy \eqref{Ccuin1p4}.

To prove the second statement, suppose that $\sum_{j \in J} a_{ij} \hat{x}_j \geq b_i$ for all $i\in I$ and we are given $\bGamma \in \bTheta$. Then it is straightforward to check that $(\bu^*_i,\by^*_i,z^*_i)$ from~\eqref{Ccu_uyz}, for all $i\in I$, satisfies \eqref{Ccuin1p1}-\eqref{Ccuin1u}.
\end{proof}

\begin{proof}[Proof of Proposition \ref{pro:CcuTwoFeas}]
$(\Rightarrow)$ The proof of this implication is divided into two cases. First, we show that feasibility of RLO-CCU-DG implies that $\sum_{j \in J} a_{ij} \hat{x}_j \geq b_i$ for all $i\in I$. We have
\begin{align*}
\sum_{j \in J} a_{ij} \hat{x}_j \geq \sum_{j \in J} a_{ij} \hat{x}_j - (\sum_{j\in J_i}y_{ij} + \Gamma_i z_i) \geq b_i, \quad \forall i\in I,
\end{align*}
where the first inequality is implied by non-negativity of $\bGamma,\by,\bz$, and the second inequality is constraint \eqref{Ccuin1p4}.

Second, assume that $\bTheta\cap\bOmega = \varnothing$. Constraint \eqref{Ccuin2side} states that $\bGamma\in\bOmega$, and by Lemma \ref{lem:CcuPfEq}, constraints \eqref{Ccuin1p1}-\eqref{Ccuin1u} imply $\bGamma\in\bTheta$, therefore RLO-CCU-DG is infeasible.

$(\Leftarrow)$ Assume that $\sum_{j \in J} a_{ij} \hat{x}_j \geq b_i$ for all $i\in I$. Let $\hat{i}\in I$ be an arbitrary index, and let $\bGamma$ be an arbitrary element of $\bTheta\cap\bOmega$. Then, it can be checked that the following is a feasible solution for RLO-CCU-DG:
\begin{subequations}\label{CcuFeasSol}
\begin{align}
\bpi &= \beee_{\hat{i}}, \\
\varphi_{ij^i_k(\bxh)}
& = \left\{
\begin{array}{ll}
\pi_i & \text{if } k=1,\dots,\lfloor\Gamma_i\rfloor, \\
(\Gamma_i - \lfloor\Gamma_i\rfloor)\pi_i & \text{if } k=\lfloor\Gamma_i\rfloor + 1, \\
0 & \text{otherwise},
\end{array}
\right. \forall \; j^i_k(\bxh)\in J_i, i\in I, \label{Ccu2in2optphi} \\
%
%\bvarphi &\text{ as in } \eqref{Ccu2in2optphi}, \\
(\lambda_{ij},\mu_{ij}) &= \left\{
\begin{array}{ll}
(\varphi_{ij},0) & \text{if } \hat{x}_j \le 0, i=\hat{i}, \\
(0,\varphi_{ij}) & \text{if } \hat{x}_j > 0, i=\hat{i}, \\
(0,0) & \text{otherwise},
\end{array}
\right. \quad \forall j\in J_i, i\in I, \\
(\bu, \by, \bz) &\text{ as in } \eqref{Ccu_uyz}, \\
c_j & = \bar{a}_{\hat{i}j}(\Gamma_{\hat{i}}, \bxh), \quad \forall j\in J.
\end{align}
\end{subequations}
\end{proof}

\begin{proof}[Proof of Theorem \ref{thm:CcuDg}]
By Lemma \ref{lem:CcuPfEq}, constraints \eqref{Ccuin1p1}-\eqref{Ccuin1u} (and thus variables $\bu, \by, \bz$) in RLO-CCU-DG can be replaced by $\bGamma \in \bTheta$. We now omit the details of several steps that are conceptually similar to steps in the proof of Theorem \ref{thm:IuDg}: we eliminate $\bc$ by substituting constraint \eqref{Ccuin1d1} into the objective function of RLO-CCU-DG, we let $s_{ij} = \lambda_{ij} - \mu_{ij}$ for all $i\in I, j\in J_i$, we use Lemma \ref{lem:sij} (see below) to replace constraints on $(\blambda, \bmu)$ with constraints on $\bs$, and then we eliminate $\bs$ by identifying its optimal solution by inspection, giving us the following optimization problem equivalent to RLO-CCU-DG:
\begin{subequations}\label{Ccu2in2eq2}
\begin{align}
\minimize_{\bGamma,\bpi,\bvarphi} \;\;\; & \sum_{i\in I}\left(\sum_{j\in J}a_{ij}\hat{x}_j - b_i\right) \pi_i - \sum_{i\in I}\sum_{j\in J_i}\alpha_{ij}|\hat{x}_j| \varphi_{ij} \\
\mbox{subject to} \;\;\; & \bGamma \in \bTheta \cap \bOmega, \label{Ccu2in2eqside} \\
& \sum_{i\in I}\pi_i = 1, \label{Ccu2in2eqnorm} \\
& 0 \le \varphi_{ij} \le \pi_i, \quad\forall j\in J_i, i\in I, \label{Ccu2in2eqd2} \\
& \sum_{j\in J_i}\varphi_{ij} \le \Gamma_i \pi_i, \quad \forall i\in I, \label{Ccu2in2eqd3} \\
& {\pi}_i \geq 0, \quad \forall i\in I. \label{Ccu2in2eqd5}
\end{align}
\end{subequations}

By inspection we can describe the optimal $\bvarphi$ for a given value of $(\bGamma,\bpi)$. In the objective function, $\bvarphi$ only appears in the term $- \sum_{i\in I}\sum_{j\in J_i}\alpha_{ij}|\hat{x}_j| \varphi_{ij}$, and the only constraints applicable to $\bvarphi$ are \eqref{Ccu2in2eqd2}-\eqref{Ccu2in2eqd3}. This is an instance of the continuous knapsack problem, so an optimal $\bvarphi$ has the form of \eqref{Ccu2in2optphi}. Substituting \eqref{Ccu2in2optphi} into formulation \eqref{Ccu2in2eq2} we obtain an equivalent optimization problem:
\begin{align}\label{Ccu2in2eq3}
\min_{\bGamma,\bpi} \left\{ \sum_{i\in I}\pi_i l_i(\bGamma, \bxh) \colon \eqref{Ccu2in2eqside}-\eqref{Ccu2in2eqnorm}, \eqref{Ccu2in2eqd5} \right\},
\end{align}
where
\begin{align*}
l_i(\bGamma, \bxh) = \sum_{j\in J}a_{ij}\hat{x}_j - b_i  - \sum_{k=1}^{\lfloor\Gamma_i\rfloor}\alpha_{ij^i_k(\bxh)}\lvert\hat{x}_{j^i_k(\bxh)}\rvert - (\Gamma_i - \lfloor\Gamma_i\rfloor)\alpha_{ij^i_{\lceil\Gamma_i\rceil}(\bxh)}\lvert\hat{x}_{j^i_{\lceil\Gamma_i\rceil}(\bxh)}\rvert,
\end{align*}
which is the surplus of constraint $i$ of the forward problem with respect to $\bxh$.

By inspection, we can see that for a given value of $\bGamma$, an optimal $\bpi$ for formulation \eqref{Ccu2in2eq3} equals $\beee_{i^*}$ where $i^*\in\argmin_{i\in I} l_i(\bGamma, \bxh)$. Note that we obtained \eqref{Ccu2in2eq3} from \eqref{Ccu2in2eq2} by setting $\bvarphi$ optimally given any feasible $(\bGamma,\bpi)$. Hence any $\bpi$ that is optimal for \eqref{Ccu2in2eq3} must also be optimal for \eqref{Ccu2in2eq2}. This means that, without changing the optimal value, we can restrict $\bpi$ to be binary in problem \eqref{Ccu2in2eq2}. The resulting problem is equivalent to $\min_{i\in I} \left\{ \eqref{CcuMipt} \right\}$. By definition, $(\bGamma^{(i)}, \bvarphi^{(i)})$ is an optimal solution for \eqref{CcuMipt}, and the optimal value of the outer problem is $\min_{i\in I} \{t_i \}$. Finally, $\bpi = \beee_{i^*}$, \eqref{Ccuin1d1}, and the definition and optimal value of $\bs$ imply that $\bc = \bab_{i^*}(\Gamma^*_{i^*}, \bxh)$.

By Assumption \ref{assumption:ccu_trivial} and the definition of $\bar{a}_{ij}$, any feasible $\bGamma$ satisfies $\bab_i(\Gamma_i, \bx) \ne \bzero$ for all $i\in I, \bx\in\mathbb{R}^n$, and therefore $\bc \ne \bzero$.
\end{proof}

\begin{proof}[Proof of Proposition \ref{pro:CcuFeas}]
$(\Rightarrow)$ The proof of this implication is divided into two cases. First, it can be shown that feasibility of RLO-CCU-SD implies that $\sum_{j \in J} a_{ij} \hat{x}_j \geq b_i$ for all $i\in I$, using the same argument as in the proof of Proposition \ref{pro:CcuTwoFeas}.

Second, assume that $\hat{I} = \varnothing$. We will show that RLO-CCU-SD is infeasible. Since $\hat{I} = \varnothing$, we have
\begin{align*}
0 & < \sum_{j \in J}a_{ij} \hat{x}_{j} - b_i - \sum_{j\in J_i}\alpha_{ij}|\hat{x}_{j}|, \quad\forall i\in I, \\
0 & < \sum_{j \in J}a_{ij} \hat{x}_{j} - b_i - \sum_{j\in J_i}\alpha_{ij}\hat{x}_{j}, \quad\forall i\in I, \\
0 & < \sum_{i\in I}\pi_i(\sum_{j \in J}a_{ij} \hat{x}_{j} - b_i - \sum_{j\in J_i}\alpha_{ij}\hat{x}_{j}) \\
& \le \sum_{i\in I} \pi_i (\sum_{j \in J}a_{ij} \hat{x}_{j} - b_i) - \sum_{i\in I}\sum_{j\in J_i}\varphi_{ij}\alpha_{ij}\hat{x}_{j} \\
& \le \sum_{i\in I} \pi_i (\sum_{j \in J}a_{ij} \hat{x}_{j} - b_i) + \sum_{i\in I}\sum_{j\in J_i}(\lambda_{ij} - {\mu}_{ij})\alpha_{ij}\hat{x}_{j},
\end{align*}
where the third inequality is implied by $\beee^\intercal\bpi = 1, \bpi \geq \bzero$; the fourth inequality is implied by constraint \eqref{Ccuin1d2}; and the fifth inequality is implied by $-\varphi_{ij} \le (\lambda_{ij} - {\mu}_{ij})$, itself implied by constraints \eqref{Ccuin1d3} and \eqref{Ccuin1d5}. Now substituting \eqref{Ccuin1d1} into \eqref{Ccuin1sd},
\begin{align*}
\sum_{i\in I} \pi_i (\sum_{j \in J}a_{ij} \hat{x}_{j} - b_i) + \sum_{i\in I}\sum_{j\in J_i}(\lambda_{ij} - {\mu}_{ij})\alpha_{ij}\hat{x}_{j} = 0,
\end{align*}
which is a contradiction.

$(\Leftarrow)$ Assume that $\sum_{j \in J} a_{ij} \hat{x}_j \geq b_i$ for all $i\in I$ and $\hat{I} \ne \varnothing$. Let $\hat{i}\in\hat{I}$ be an arbitrary index. Then, it can be checked that the following is a feasible solution for RLO-CCU-SD: $(\bc,\bu,\by,\bz,\bpi,\bvarphi,\blambda,\bmu)$ as in \eqref{CcuFeasSol}, and
\begin{align*}
%\bpi &= \beee_{\hat{i}}, \\
%\bvarphi &\text{ as in } \eqref{Ccu2in2optphi}, \\
%\varphi_{i{j^i_k}}
%& = \left\{
%\begin{array}{ll}
%\pi_i & \text{if } k=1,\dots,\lfloor\Gamma_i\rfloor, i=\hat{i}, \\
%(\Gamma_i - \lfloor\Gamma_i\rfloor)\pi_i & \text{if } k=\lceil\Gamma_i\rceil, i=\hat{i}, \\
%0 & \text{otherwise}, \\
%\end{array}
%\right. \\
%(\lambda_{ij},\mu_{ij}) &= \left\{
%\begin{array}{ll}
%(\varphi_{ij},0) & \text{if } \hat{x}_j \le 0, i=\hat{i}, \\
%(0,\varphi_{ij}) & \text{if } \hat{x}_j > 0, i=\hat{i}, \\
%(0,0) & \text{otherwise},
%\end{array}
%\right. \quad \forall j\in J_i, i\in I, \\
\Gamma_i &= \left\{
\begin{array}{ll}
\underline{\Gamma}_i & \text{if } i=\hat{i}, \\
0 & \text{if } i\in I\setminus\{\hat{i}\},
\end{array}
\right. %\\
%u_{ij} &= \alpha_{ij}\lvert\hat{x}_j\rvert, \quad \forall j\in J_i, i\in I, \\
%y_{ij}
%& = \left\{
%\begin{array}{ll}
%\max\{u_{ij}-z_i,0\} & \text{if } i=\hat{i}, \\
%0 & \text{otherwise}, \\
%\end{array}
%\right. \quad \forall j\in J_i, i\in I, \\
%z_i &= \left\{
%\begin{array}{ll}
%\alpha_{i j^i_{\lceil\Gamma_i\rceil}(\bxh)} \lvert\hat{x}_{j^i_{\lceil\Gamma_i\rceil}(\bxh)}\rvert & \text{if } J_i\ne\varnothing, i=\hat{i}, \\
%\max_{j\in J_i}\{\alpha_{ij}|\hat{x}_j|\} & \text{if } J_i\ne\varnothing, i\in I\setminus\{\hat{i}\}, \\
%0 & \text{if } J_i = \varnothing, \\
%\end{array}
%\right. \quad \forall i\in I, \\
%c_j & = \bar{a}_{\hat{i}j}(\Gamma_{\hat{i}}, \bxh), \quad \forall j\in J.
%c^i_j & = \left\{
%%
%\begin{array}{ll}
%a_{ij} - \sgn({\hat{x}_j})\alpha_{ij} & \text{if } j = j^i_k, k = 1, \dots, \lfloor\Gamma_i\rfloor, \\
%%
%a_{ij} - \sgn({\hat{x}_j})\alpha_{ij}(\Gamma_i - \lfloor\Gamma_i\rfloor) & \text{if } j = j^i_{\lfloor\Gamma_i\rfloor +1 }, \\
%%
%a_{ij} & \text{otherwise}, \quad \forall i\in \hat{I}.
%\end{array}
%\right.
\end{align*}
\end{proof}

\begin{proof}[Proof of Lemma \ref{lem:CcuFeas2}]
To prove the first statement, we first note that by Lemma \ref{lem:CcuPfEq}, constraints \eqref{Ccuin1p1}-\eqref{Ccuin1u} imply $\bGamma \in \bTheta$. So we only need to show that the constraints of RLO-CCU-SD imply $\Gamma_{\hat{i}} = \underline{\Gamma}_{\hat{i}}$, for some $\hat{i}\in\hat{I}$. If we substitute \eqref{Ccuin1d1} into \eqref{Ccuin1sd}, let $s_{ij} = \lambda_{ij} - \mu_{ij}$, and use reasoning similar to the proof of Lemma \ref{lem:sij} (see below), we get
\begin{align*}
\sum_{i\in I} \pi_i \left(\sum_{j \in J}a_{ij} \hat{x}_{j} - b_i\right) \le \sum_{i\in I}\sum_{j\in J_i}\alpha_{ij}|\hat{x}_{j}|\varphi_{ij}.
\end{align*}
For a given $(\bGamma,\bpi)$, the constraints applicable to $\bvarphi$ are \eqref{Ccuin1d2}, \eqref{Ccuin1d4}, and nonnegativity. As such, the maximum value of $\sum_{j\in J_i}\alpha_{ij}|\hat{x}_{j}|\varphi_{ij}$ over feasible $\bvarphi_i$, for all $i\in I$, is the optimal value of the following optimization problem:
\begin{equation}
\begin{aligned}\label{knapsack}
\maximize_{\bvarphi_i} \;\;\; & \sum_{j\in J_i}\alpha_{ij}|\hat{x}_{j}|\varphi_{ij} \\
\mbox{subject to} \;\;\; & 0 \le \varphi_{ij} \le \pi_i, \;\forall j\in J_i, \\
& \sum_{j\in J_i}\varphi_{ij} \le \Gamma_i\pi_i.
\end{aligned}
\end{equation}
Formulation~\eqref{knapsack} is an instance of the continuous knapsack problem, equivalent to the one which appears in the proof of Theorem \ref{thm:CcuDg}, so its optimal value is
\begin{align*}
\sum_{k=1}^{\lfloor\Gamma_i\rfloor}\alpha_{ij^i_k(\bxh)}|\hat{x}_{j^i_k(\bxh)}|\pi_i + (\Gamma_i - \lfloor\Gamma_i\rfloor)\alpha_{ij^i_{\lceil\Gamma_i\rceil}(\bxh)}|\hat{x}_{j^i_{\lceil\Gamma_i\rceil}(\bxh)}|\pi_i,
\end{align*}
and we can conclude that
\begin{align}\label{temp5}
\sum_{i\in I} \pi_i \left(\sum_{j \in J}a_{ij} \hat{x}_{j} - b_i\right) \le \sum_{i\in I}\pi_i \left( \sum_{k=1}^{\lfloor\Gamma_i\rfloor}\alpha_{ij^i_k(\bxh)}|\hat{x}_{j^i_k(\bxh)}| + (\Gamma_i - \lfloor\Gamma_i\rfloor)\alpha_{ij^i_{\lceil\Gamma_i\rceil}(\bxh)}|\hat{x}_{j^i_{\lceil\Gamma_i\rceil}(\bxh)}| \right).
\end{align}
Now, assume to the contrary that $\Gamma_i < \underline{\Gamma}_i$ for all $i\in\hat{I}$. Using this assumption and the fact that $\underline{\Gamma}_i$ is the smallest value of $\Gamma_i$ to satisfy \eqref{gammabar}, we can deduce a contradiction with \eqref{temp5}.

To prove the second statement, assume $\sum_{j \in J} a_{ij} \hat{x}_j \geq b_i$ for all $i\in I$ and let $(\bc,\bu,\by,\bz,\bpi,\bvarphi,\blambda,\bmu)$ be defined as in the proof of Proposition \ref{pro:CcuFeas}. With the assumed $\bGamma$, this solution is feasible for RLO-CCU-SD.
\end{proof}

\begin{proof}[Proof of Theorem \ref{thm:CcuAlg}]
By Lemma \ref{lem:CcuFeas2}, solving RLO-CCU-SD is equivalent to solving the following optimization problem for all $\hat{i}\in\hat{I}$, and taking the minimum over all $\lvert\hat{I}\rvert$ optimal values:
\begin{equation}
\begin{aligned}\label{Ccuin1eq}
\minimize_{\bGamma} \;\;\; & \lVert\bGamma - \bGammah \rVert \\
\mbox{subject to} \;\;\; & \Gamma_{\hat{i}} = \underline{\Gamma}_{\hat{i}}, \\
& 0 \le \Gamma_i \le \underline\Gamma_i, \quad\forall i \in \hat{I}, \\
& 0 \le \Gamma_i \le \lvert J_i\rvert, \quad\forall i \in I\setminus\hat{I},
\end{aligned}
\end{equation}
where the cost vector corresponding to the $\hat{i}$-th formulation is $\bc = \bab_i(\Gamma_i, \bxh)$, as described in the proof of Lemma \ref{lem:CcuFeas2}.

Suppose we consider formulation \eqref{Ccuin1eq} for some fixed $\hat{i}\in\hat{I}$. For all $i\in I$, the variable $\Gamma_i$ is included only in a term $(\Gamma_i - \hat{\Gamma}_i)$ in the objective function. For $i\in I\setminus\hat{I}$, because $0 \le \hat{\Gamma}_i \le |J_i|$ by Assumption \ref{assumption:gammabounds}, the optimal solution is $\Gamma_i = \hat{\Gamma}_i$, so $\Gamma_i - \hat{\Gamma}_i = 0 = g_i$. For $i\in\hat{I}\setminus\{\hat{i}\}$, we will have either $0 \le \hat{\Gamma}_i \le |\underline{\Gamma}_i|$ or $|\underline{\Gamma}_i| < \hat{\Gamma}_i \le |J_i|$ by Assumption \ref{assumption:gammabounds}, hence the optimal solution is $\Gamma_i = \min\{\hat{\Gamma}_i, \underline{\Gamma}_i\}$, and then $\Gamma_i - \hat{\Gamma}_i = g_i$. For $i=\hat{i}$, we require $\Gamma_{\hat{i}} = \underline{\Gamma}_{\hat{i}}$, so $\Gamma_{\hat{i}} - \hat{\Gamma}_{\hat{i}} = f_{\hat{i}}$.

In other words, all objective function terms $(\Gamma_i - \hat{\Gamma}_i)$ equal $g_i$, except $(\Gamma_{\hat{i}} - \hat{\Gamma}_{\hat{i}}) = f_{\hat{i}}$. It follows that the optimal value of the $\hat{i}$-th formulation~\eqref{Ccuin1eq} is $\lVert\bg + (f_{\hat{i}} - g_{\hat{i}})\beee_{\hat{i}}\rVert$. Thus the optimal value of RLO-CCU-SD is $\min_{\hat{i}\in \hat{I}} \left\{ \lVert\bg + (f_{\hat{i}} - g_{\hat{i}})\beee_{\hat{i}}\rVert \right\}$. Finally, by the same reasoning as in Theorem \ref{thm:CcuDg}, the optimal solution \eqref{CcuTwogsol}-\eqref{CcuTwocsol} satisfies $\bc \ne \bzero$ and $\bab_i(\Gamma_i,\bx) \ne \bzero$ for all $i\in I, \bx\in\mathbb{R}^n$.
\end{proof}

\begin{lem}
\label{lem:IuFeas2}
Every feasible solution for RLO-IU-SD satisfies \eqref{Iuin1eq2}-\eqref{Iuin1eq4} for some $\hat{i}\in I$. Conversely, for every $\balpha$ satisfying \eqref{Iuin1eq2}-\eqref{Iuin1eq4} for some $\hat{i}\in I$, there exists $(\bc,\bu,\bpi,\blambda,\bmu)$ such that $(\balpha,\bc,\bu,\bpi,\blambda,\bmu)$ is feasible for RLO-IU-SD.
\end{lem}
\begin{proof}To prove the first statement, we first note that constraints \eqref{Iuin1p1}-\eqref{Iuin1u} imply \eqref{Iuin1eq3}. To complete the proof, we assume $\sum_{j\in J}a_{ij} \hat{x}_j -\sum_{j\in J_i}{{\alpha}_{ij}|\hat{x}_{j}|} > b_i$ for all $i\in I$ and derive a contradiction, following similar reasoning as in the proof of Lemma \ref{lem:LpFeas2}. Substituting \eqref{Iuin1d1} into \eqref{Iuin1sd}, we get
\begin{equation*}
\sum_{i\in I} \pi_i \sum_{j \in J}a_{ij} \hat{x}_{j} + \sum_{i\in I}\sum_{j\in J_i}(\lambda_{ij} - {\mu}_{ij})\alpha_{ij}\hat{x}_{j} = \sum_{i\in I} \pi_i b_i.
\end{equation*}
Let $s_{ij} = \lambda_{ij} - \mu_{ij}$ for all $j\in J_i, i\in I$. By Lemma \ref{lem:sij} (see below), $(\blambda,\bmu)$ satisfies \eqref{Iuin1d2} and \eqref{Iuin1d3} if and only if $s_{ij} \in [-\pi_i, \pi_i]$ for all $j\in J_i, i\in I$. Thus, the feasible region of RLO-IU-SD is equivalent to
\begin{subequations}\label{Iuin1eq2a}
\begin{align}
& \sum_{i\in I} \pi_i \sum_{j \in J}a_{ij} \hat{x}_{j} + \sum_{i\in I}\sum_{j\in J_i}s_{ij}\alpha_{ij}\hat{x}_{j} = \sum_{i\in I} \pi_i b_i, \label{temp1} \\
& -\pi_i \le s_{ij} \le \pi_i, \quad \forall j\in J_i, i\in I, \label{temp2} \\
& \beee^\intercal\bpi = 1, \bpi \geq \bzero, \label{temp} \\
& \eqref{Iuin1p1}-\eqref{Iuin1u}.
\end{align}
\end{subequations}

Constraint \eqref{temp} ensures that $\overline{I} := \{i\in I \colon \pi_i > 0\} \ne \varnothing$. Since we have assumed $\sum_{j \in J} a_{ij} \hat{x}_{j} - \sum_{j\in J_{i}}\alpha_{ij}|\hat{x}_{j}| > b_{i}$ for all $i\in I$, we have
\begin{align*}
\pi_{i}\sum_{j \in J} a_{ij} \hat{x}_{j} - \pi_{i}\sum_{j\in J_{i}}\alpha_{ij}|\hat{x}_{j}| > \pi_{i}b_{i}, \quad \forall i \in \overline{I}.
\end{align*}
For all $j\in J_i, i\in I$, $s_{ij} \in [-\pi_i,\pi_i]$ implies that $\sum_{j\in J_{i}}s_{ij}\alpha_{ij}\hat{x}_{j} \geq - \pi_{i}\sum_{j\in J_{i}}\alpha_{ij}|\hat{x}_{j}|$, and therefore
\begin{align*}
\pi_{i}\sum_{j \in J} a_{ij} \hat{x}_{j} + \sum_{j\in J_{i}}s_{ij}\alpha_{ij}\hat{x}_{j} > \pi_{i}b_{i}, \quad \forall i \in \overline{I}.
\end{align*}
Since $s_{ij} = 0$ if $\pi_i = 0$,
\begin{align*}
\sum_{i\in I}\pi_{i}\sum_{j \in J} a_{ij} \hat{x}_{j} + \sum_{i\in I}\sum_{j\in J_{i}}s_{ij}\alpha_{ij}\hat{x}_{j} > \sum_{i\in I}\pi_{i} b_{i},
\end{align*}
which contradicts constraint \eqref{temp1}.

To prove the second statement, let $\balpha$ satisfy \eqref{Iuin1eq2}-\eqref{Iuin1eq4}, $\hat{i}$ be defined by \eqref{Iuin1eq2}, and $(\bc,\bu,\bpi,\blambda,\bmu)$ be defined as in the proof of Proposition \ref{pro:IuFeas}. This solution is feasible for RLO-IU-SD.
\end{proof}

\begin{lem}\label{lem:sij}
Let $s_{ij} = \lambda_{ij} - \mu_{ij}$ for all $j\in J_i, i\in I$. If $(\blambda,\bmu)$ satisfies \eqref{Iuin1d2} and non-negativity, then $s_{ij} \in [-\pi_i, \pi_i]$ for all $j\in J_i, i\in I$. Conversely, if $s_{ij} \in [-\pi_i, \pi_i]$ for all $j\in J_i, i\in I$, then there exists $(\blambda,\bmu)$ satisfying \eqref{Iuin1d2}, non-negativity, and $s_{ij} = \lambda_{ij} - \mu_{ij}$ for all $j\in J_i, i\in I$.
\end{lem}
\begin{proof} To prove the first statement, note that since $\lambda_{ij} + \mu_{ij} = \pi_i$ and $\lambda_{ij}, \mu_{ij}, \pi_i \geq 0$, it follows that $\lambda_{ij} \le \pi_i$, for all $j\in J_i, i\in I$. Since $\mu_{ij} \geq 0$, it further follows that $\lambda_{ij} - \mu_{ij} \le \pi_i$, i.e., $s_{ij} \le \pi_i$. The proof of $-\pi_i \le s_{ij}$ is similar.

To prove the second statement, we will construct $(\blambda,\bmu)$ satisfying the required conditions. For all $j\in J_i, i\in I$, let
\begin{align*}
(\lambda_{ij}, \mu_{ij}) = \left\{
\begin{array}{ll}
\left(s_{ij} + \frac{\pi_i-s_{ij}}{2}, \frac{\pi_i-s_{ij}}{2}\right) & \text{if } s_{ij} \geq 0, \\
\left(\frac{\pi_i+s_{ij}}{2}, -s_{ij} + \frac{\pi_i+s_{ij}}{2}\right) & \text{otherwise.}
\end{array}
\right.
\end{align*}
\end{proof}

\begin{lem}
\label{lem:aux}
The solution $(\bu^*_i,\by^*_i,z^*_i)$ defined in \eqref{Ccu_uyz} is an optimal solution to \eqref{aux}.
\end{lem}
\begin{proof} By inspection, it is clear that an optimal solution has $u^*_{ij}=\alpha_{ij} \lvert\hat{x}_{j}\rvert$ and $y^*_{ij}=\max\{u^*_{ij}-z_i,0\}$ for a given $z_i$, so we only need to prove that $z_i^*=\alpha_{ij^i_{\lceil\Gamma_i\rceil}(\bxh)} \lvert\hat{x}_{j^i_{\lceil\Gamma_i\rceil}(\bxh)}\rvert$ is optimal. Let $f(\bu_i, \by_i, z_i)$ denote the objective function of \eqref{aux}. We have
\begin{align*}
f(\bu^*_i,\by^*_i,z^*_i) &= \sum_{j\in J_i}\max\{\alpha_{ij} |\hat{x}_j| - z_i^*, 0\} + \Gamma_i z_i^* \\
&= \sum_{k=1}^{\lceil\Gamma_i\rceil - 1}(\alpha_{ij^i_k(\bxh)}|\hat{x}_{ij^i_k(\bxh)}| - z_i^*) + \Gamma_i z_i^*,
\end{align*}
and we will show that $f(\bu^*_i,\by^*_i,\tilde{z}_i) \geq f(\bu^*_i,\by^*_i,z^*_i)$ for any $\tilde{z}_i \ne z_i^*$. We define $\hat{k} = |\{j\in J_i\colon \alpha_{ij} \lvert\hat{x}_{j}\rvert > \tilde{z}_i\}|$.

We distinguish two cases, and first consider the case $\tilde{z}_i > z_i^*$. In this case,
\begin{align*}
f(\bu^*_i,\by^*_i,z^*_i) &= \sum_{k=1}^{\hat{k}}(\alpha_{ij^i_k(\bxh)}|\hat{x}_{ij^i_k(\bxh)}| - z_i^*) + \sum_{k=\hat{k}+1}^{\lceil\Gamma_i\rceil - 1}(\alpha_{ij^i_k(\bxh)}|\hat{x}_{ij^i_k(\bxh)}| - z_i^*) + \Gamma_i z_i^*, \\
\intertext{and}
f(\bu^*_i,\by^*_i,\tilde{z}_i) &= \sum_{j\in J_i}\max\{\alpha_{ij} |\hat{x}_j| - \tilde{z}_i, 0\} + \Gamma_i \tilde{z}_i \\
&= \sum_{k=1}^{\hat{k}}(\alpha_{ij^i_k(\bxh)}|\hat{x}_{ij^i_k(\bxh)}| - \tilde{z}_i) + \Gamma_i \tilde{z}_i.
\end{align*}
Then
\begin{align*}
& f(\bu^*_i,\by^*_i,\tilde{z}_i) - f(\bu^*_i,\by^*_i,z^*_i) \\
=\; & \hat{k}(z^*_i - \tilde{z}_i) - \sum_{k=\hat{k}+1}^{\lceil\Gamma_i\rceil - 1}\alpha_{ij^i_k(\bxh)}|\hat{x}_{ij^i_k(\bxh)}| + (\lceil\Gamma_i\rceil - 1 - \hat{k})z^*_i + \Gamma_i(\tilde{z}_i - z^*_i) \\
=\; & (\Gamma_i - \hat{k})\tilde{z}_i - \sum_{k=\hat{k}+1}^{\lceil\Gamma_i\rceil - 1}\alpha_{ij^i_k(\bxh)}|\hat{x}_{ij^i_k(\bxh)}| - \left(\Gamma_i - (\lceil\Gamma_i\rceil - 1)\right)\alpha_{ij^i_{\lceil\Gamma_i\rceil}(\bxh)} |\hat{x}_{j^i_{\lceil\Gamma_i\rceil}(\bxh)}|
\end{align*}
which is greater than or equal to 0 since $\tilde{z}_i \geq \alpha_{ij^i_k(\bxh)} |\hat{x}_{j^i_k(\bxh)}|$ for all $k = \hat{k} + 1, \dots, \lceil\Gamma_i\rceil$. The proof for the case $z^*_i > \tilde{z}_i$ is similar.
\end{proof}

\section{Numerical examples of NLO-SD with Assumption \ref{assumption:nominal_trivial} violated}\label{sec:a2}

In this section, we provide numerical examples to illustrate that it is possible to circumvent a trivial solution for NLO-SD that may be induced when Assumption \ref{assumption:nominal_trivial} is not satisfied. Of the two conditions required by the assumption for all $i\in I$, $b_i\ne 0$ can be considered strong, whereas $\bah_i \ne \bzero$ can be considered mild, so we focus on the case where the former requirement is not satisfied. In this case, we propose three possible ways to avert a trivial solution: perturbing either $b_i$, $\bah_i$, or $\xi_i$.

\begin{exmp}[\textbf{NLO-SD with trivial cost vector and constraint}]\label{ex:LpTrivialCost}

Let the norm in the objective function be the Euclidean norm and let $\bxi = \beee$. Let the observed solution be $\bxh = (2, 2)$, and let the remaining problem data be
\begin{align*}
\bAh =
\begin{pmatrix}
1 & 0 \\
0 & 1 \\
1 & 1 \\
-1 & -1
\end{pmatrix}, \quad
\bb =
\begin{pmatrix}
-3 \\
-3 \\
0 \\
-10
\end{pmatrix}.
\end{align*}
The prior feasible region is illustrated in Figure \ref{fig:LpTrivialCost1}. The observed solution is an interior point of the prior feasible region, so the IO model is expected to adjust a single constraint such that it becomes active. Applying Theorem \ref{thm:LpAlg}, we find $\bg = \bzero$ and $\bff = (1.77,1.77,1.41,2.12)$, thus $i^* = 3$. However, $b_3 = 0$ violates Assumption \ref{assumption:nominal_trivial}, and in combination with the given values of $\bah_3$ and $\bxh$, the result is that $\ba_3 = \ba^f_3 = (0,0)$ and $\bc = \bzero$. This solution effectively means that although the IO model was supposed to adjust the third constraint to become active, it does so artificially by setting all coefficients equal to zero, effectively eliminating the constraint and implying a zero cost vector, as shown in Figure \ref{fig:LpTrivialCost1}. The observed $\bxh$ remains an interior point, thus is not optimal with respect to any nonzero cost vector.

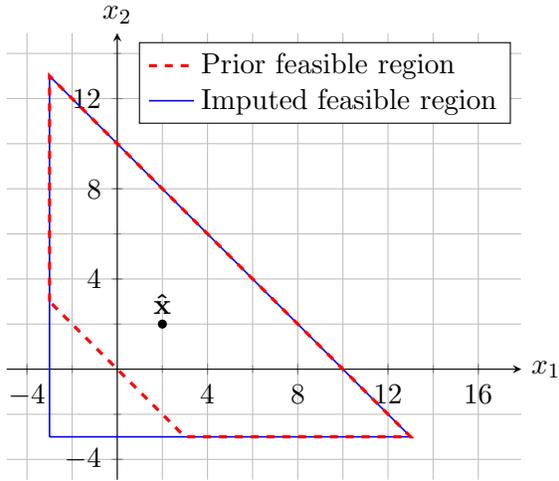
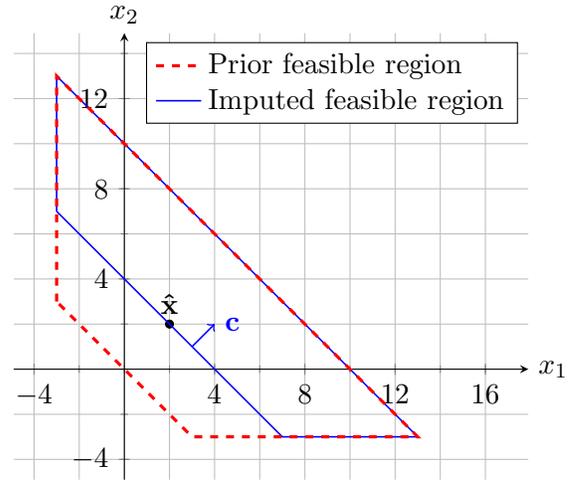
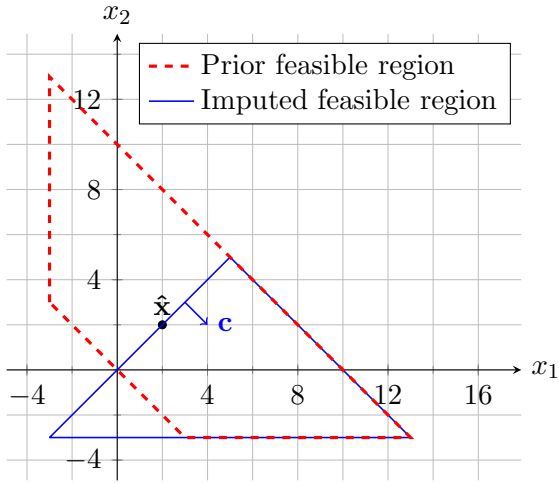
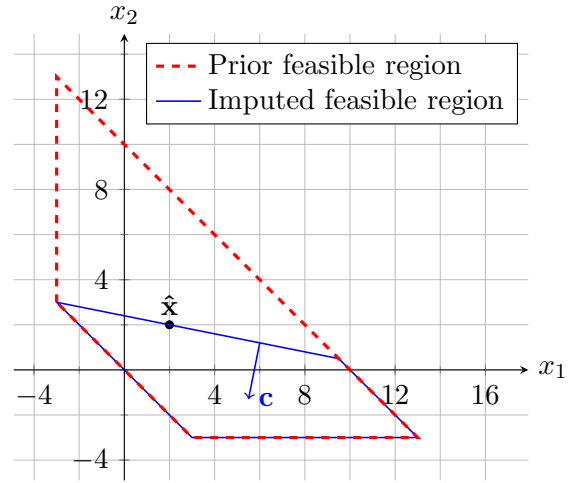
\begin{figure*}
\centering
\begin{subfigure}{0.5\textwidth}
\centering
\begin{tikzpicture}
		\begin{axis}[
			axis x line=center,
			axis y line=center,
			x=0.3cm, y=0.3cm, %sets the axis units, effectively controlling size of picture
			xlabel={$x_1$},
			ylabel={$x_2$},
			xlabel style={right},
			ylabel style={above},
			xtick={-4,4,8,12,16},
			ytick={-4,4,8,12},
			minor xtick={-4,-2,0,2,...,16}, %for grid
			minor ytick={-4,-2,...,14}, %for grid
			grid=minor,
			legend cell align={left},
			xmin=-4.9,
			xmax=17.9,
			ymin=-4.9,
			ymax=14.9]
				
	  \filldraw (2,2) circle (1.5pt) node[above] {$\bxh$}; % observed solution
		\draw[semithick,blue] (-3,-3) -- (-3,13) -- (13,-3) -- (-3,-3); % imputed feasible region
		\draw[very thick,red,dashed] (-3,3) -- (-3,13) -- (13,-3) -- (3,-3) -- (-3,3); % prior feasible region
		
		% Legend
		\addlegendimage{very thick,red,dashed}
		\addlegendentry{Prior feasible region}
		\addlegendimage{semithick,blue}
		\addlegendentry{Imputed feasible region}
		
		\end{axis}
\end{tikzpicture}
\caption{Trivial cost vector and constraint induced by pathological problem data.}
\label{fig:LpTrivialCost1}
\end{subfigure}~~~~ 
\begin{subfigure}{0.5\textwidth}
\centering
\begin{tikzpicture}
		\begin{axis}[
			axis x line=center,
			axis y line=center,
			x=0.3cm, y=0.3cm, %sets the axis units, effectively controlling size of picture
			xlabel={$x_1$},
			ylabel={$x_2$},
			xlabel style={right},
			ylabel style={above},
			xtick={-4,4,8,12,16},
			ytick={-4,4,8,12},
			minor xtick={-4,-2,0,2,...,16}, %for grid
			minor ytick={-4,-2,...,14}, %for grid
			grid=minor,
			legend cell align={left}, 
			xmin=-4.9,
			xmax=17.9,
			ymin=-4.9,
			ymax=14.9]
				
	  \filldraw (2,2) circle (1.5pt) node[above] {$\bxh$}; % observed solution
		%\draw[semithick,blue] (-3,3) -- (-3,13) -- (13,-3) -- (3,-3) -- (-3,3); % imputed feasible region
		\draw[semithick,blue] (-3,7) -- (-3,13) -- (13,-3) -- (7,-3) -- (-3,7); % imputed feasible region
		\draw[very thick,red,dashed] (-3,3) -- (-3,13) -- (13,-3) -- (3,-3) -- (-3,3); % prior feasible region
		%\draw[->, semithick, blue] (0,0) -- (1,1) node[right] {$\bc$}; % cost vector
		\draw[->, semithick, blue] (3,1) -- (4,2) node[right] {$\bc$}; % cost vector
		
		% Legend
		\addlegendimage{very thick,red,dashed}
		\addlegendentry{Prior feasible region}
		\addlegendimage{semithick,blue}
		\addlegendentry{Imputed feasible region}
		
		\end{axis}
\end{tikzpicture}
\caption{Trivial solution averted with perturbed right-hand-side constraint vector.}
\label{fig:LpTrivialCost2}
\end{subfigure}
\begin{subfigure}{0.5\textwidth}
\centering
\begin{tikzpicture}
		\begin{axis}[
			axis x line=center,
			axis y line=center,
			x=0.3cm, y=0.3cm, %sets the axis units, effectively controlling size of picture
			xlabel={$x_1$},
			ylabel={$x_2$},
			xlabel style={right},
			ylabel style={above},
			xtick={-4,4,8,12,16},
			ytick={-4,4,8,12},
			minor xtick={-4,-2,0,2,...,16}, %for grid
			minor ytick={-4,-2,...,14}, %for grid
			grid=minor,
			legend cell align={left}, 
			xmin=-4.9,
			xmax=17.9,
			ymin=-4.9,
			ymax=14.9]
				
	  \filldraw (2,2) circle (1.5pt) node[above] {$\bxh$}; % observed solution
		%\draw[semithick,blue] (-3,3) -- (-3,13) -- (13,-3) -- (3,-3) -- (-3,3); % imputed feasible region
		\draw[semithick,blue] (-3,-3) -- (5,5) -- (13,-3) -- (-3,-3); % imputed feasible region
		\draw[very thick,red,dashed] (-3,3) -- (-3,13) -- (13,-3) -- (3,-3) -- (-3,3); % prior feasible region
		%\draw[->, semithick, blue] (0,0) -- (1,1) node[right] {$\bc$}; % cost vector
		\draw[->, semithick, blue] (3,3) -- (4,2) node[right] {$\bc$}; % cost vector
		
		% Legend
		\addlegendimage{very thick,red,dashed}
		\addlegendentry{Prior feasible region}
		\addlegendimage{semithick,blue}
		\addlegendentry{Imputed feasible region}
		
		\end{axis}
\end{tikzpicture}
\caption{Trivial solution averted with perturbed prior left-hand-side constraint vector.}
\label{fig:LpTrivialCost3}
\end{subfigure}~~~~ 
\begin{subfigure}{0.5\textwidth}
\centering
\begin{tikzpicture}
		\begin{axis}[
			axis x line=center,
			axis y line=center,
			x=0.3cm, y=0.3cm, %sets the axis units, effectively controlling size of picture
			xlabel={$x_1$},
			ylabel={$x_2$},
			xlabel style={right},
			ylabel style={above},
			xtick={-4,4,8,12,16},
			ytick={-4,4,8,12},
			minor xtick={-4,-2,0,2,...,16}, %for grid
			minor ytick={-4,-2,...,14}, %for grid
			grid=minor,
			legend cell align={left}, 
			xmin=-4.9,
			xmax=17.9,
			ymin=-4.9,
			ymax=14.9]
				
	  \filldraw (2,2) circle (1.5pt) node[above] {$\bxh$}; % observed solution
		%\draw[semithick,blue] (-3,3) -- (-3,13) -- (13,-3) -- (3,-3) -- (-3,3); % imputed feasible region
		\draw[semithick,blue] (-3,3) -- (9.5,0.5) -- (13,-3) -- (3,-3) -- (-3,3); % imputed feasible region
		\draw[very thick,red,dashed] (-3,3) -- (-3,13) -- (13,-3) -- (3,-3) -- (-3,3); % prior feasible region
		%\draw[->, semithick, blue] (0,0) -- (1,1) node[right] {$\bc$}; % cost vector
		\draw[->, semithick, blue] (6,1.2) -- (5.5,-1.3) node[right] {$\bc$}; % cost vector
		
		% Legend
		\addlegendimage{very thick,red,dashed}
		\addlegendentry{Prior feasible region}
		\addlegendimage{semithick,blue}
		\addlegendentry{Imputed feasible region}
		
		\end{axis}
\end{tikzpicture}
\caption{Trivial solution averted with perturbed objective function weights in NLO-SD.}
\label{fig:LpTrivialCost4}
\end{subfigure}
\caption{A numerical example in which NLO-SD produces a trivial cost vector and eliminates a constraint of the forward problem.}
\label{fig:LpTrivialCost}
\end{figure*}

We can circumvent this issue by perturbing the problem data in three different ways. First, we add a small value of 0.1 to $b_3$, so that Assumption \ref{assumption:nominal_trivial} is satisfied. We find that $\bg$, $\bff$, and $i^*$ have the same values as before. However, we find $\ba_3 = (0.025, 0.025)$, i.e., the third constraint is adjusted to $0.025x_1 + 0.025x_2 \geq 0.1$, which is rendered active by $\bxh$, and correspondingly $\bc = (0.025, 0.025)$ makes $\bxh$ optimal, as shown in Figure \ref{fig:LpTrivialCost2}. Second, we can instead add 0.1 to the prior estimate $\hat{a}_{31}$, in which case the third constraint is adjusted to $0.005x_1 - 0.005x_2 \geq 0$ (see Figure \ref{fig:LpTrivialCost3}). Third, we can increase the value of $\xi_3$ to 10 to force the IO model to adjust a constraint other than the third one: we find that the first constraint is adjusted to $-0.25x_1 - 1.25x_2 \geq -3$ (see Figure \ref{fig:LpTrivialCost4}).
\end{exmp}

\begin{exmp}[\textbf{NLO-SD with trivial constraint and non-trivial cost vector}]\label{ex:LpTrivialCon}

Let $\bxh$, $||\cdot||$, and $\bxi$ be the same as in Example \ref{ex:LpTrivialCost}, and let the constraint data be
\begin{align*}
\bAh =
\begin{pmatrix}
1 & 0 \\
0 & 1 \\
-1 & -1
\end{pmatrix}, \quad
\bb =
\begin{pmatrix}
2 \\
-4 \\
0
\end{pmatrix}.
\end{align*}
The problem data differs from the previous example in two respects. First, $\bxh$ is on the boundary of the first constraint of the prior feasible region, which causes NLO-SD to impute the non-trivial $\bc = \bah_1 = (1,0)$. Second, $\bxh$ is infeasible with respect to $\bah_3$, and thus the third constraint must be adjusted such that $\bxh$ becomes feasible. However, because $b_3 = 0$ and because of the given values of $\bah_3$ and $\bxh$, the result is that $\ba^g_3 = (0,0)$. In other words, the third constraint is eliminated and the feasible region is rendered unbounded, as shown in Figure \ref{fig:LpTrivialCon1}. This solution is meaningful insofar as $\bxh$ does lie on the boundary of the imputed feasible region and is therefore optimal, but it is trivial in the sense that it achieves feasibility for the third constraint simply by eliminating it entirely.

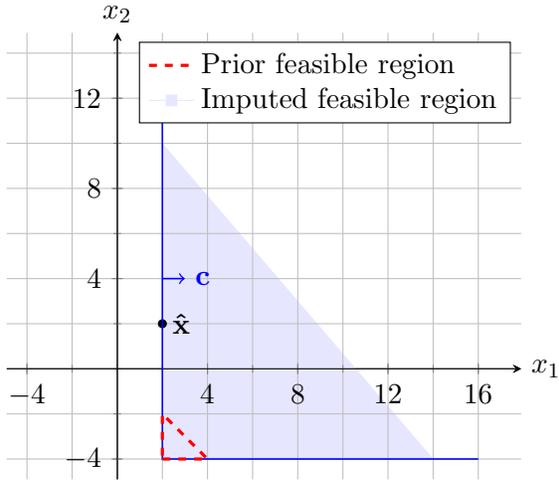
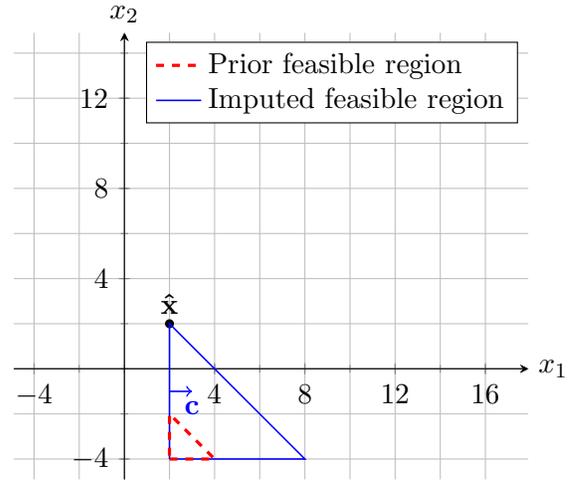
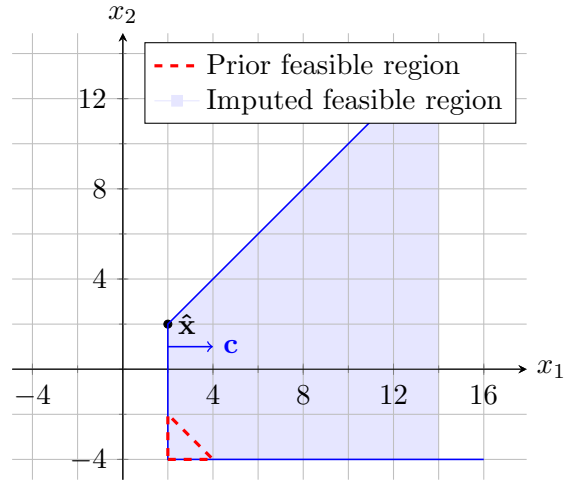
\begin{figure*}
\centering
\begin{subfigure}{0.5\textwidth}
\centering
\begin{tikzpicture}
\begin{axis}[
			axis x line=center,
			axis y line=center,
			x=0.3cm, y=0.3cm, %sets the axis units, effectively controlling size of picture
			xtick={100},
			ytick={100},
			%minor xtick={-2,0,2,...,16}, %for grid
			%minor ytick={-4,-2,...,14}, %for grid
			%grid=minor,
			%legend cell align={left}, 
			xmin=-4.9,
			xmax=17.9,
			ymin=-4.9,
			ymax=14.9]
				
		\fill[fill=blue!10!white] (2,10) -- (2,-4) -- (14,-4); % shading the imputed feasible region
		\end{axis}
		
		\begin{axis}[
			axis x line=center,
			axis y line=center,
			x=0.3cm, y=0.3cm, %sets the axis units, effectively controlling size of picture
			xlabel={$x_1$},
			ylabel={$x_2$},
			xlabel style={right},
			ylabel style={above},
			xtick={-4,4,8,12,16},
			ytick={-4,4,8,12},
			minor xtick={-4,-2,0,2,...,16}, %for grid
			minor ytick={-4,-2,...,14}, %for grid
			grid=minor,
			legend cell align={left}, 
			xmin=-4.9,
			xmax=17.9,
			ymin=-4.9,
			ymax=14.9]
				
	  \filldraw (2,2) circle (1.5pt) node[right] {$\bxh$}; % observed solution
		\draw[semithick,blue] (2,14) -- (2,-4) -- (16,-4); % imputed feasible region
		\draw[very thick,red,dashed] (2,-4) -- (4,-4) -- (2,-2) -- (2,-4); % prior feasible region
		\draw[->, semithick, blue] (2,4) -- (3,4) node[right] {$\bc$}; % cost vector
		
		% Legend
		\addlegendimage{very thick,red,dashed}
		\addlegendentry{Prior feasible region}
		\addlegendimage{mark=square*,color=blue!10!white}
		\addlegendentry{Imputed feasible region}
		
		\end{axis}
\end{tikzpicture}
\caption{Elimination of prior infeasible constraint, induced by pathological problem data.}
\label{fig:LpTrivialCon1}
\end{subfigure}~~~~ 
\begin{subfigure}{0.5\textwidth}
\centering
\begin{tikzpicture}
\begin{axis}[
			axis x line=center,
			axis y line=center,
			x=0.3cm, y=0.3cm, %sets the axis units, effectively controlling size of picture
			xtick={100},
			ytick={100},
			%minor xtick={-2,0,2,...,16}, %for grid
			%minor ytick={-4,-2,...,14}, %for grid
			%grid=minor,
			%legend cell align={left}, 
			xmin=-4.9,
			xmax=17.9,
			ymin=-4.9,
			ymax=14.9]
				
		\end{axis}
		
		\begin{axis}[
			axis x line=center,
			axis y line=center,
			x=0.3cm, y=0.3cm, %sets the axis units, effectively controlling size of picture
			xlabel={$x_1$},
			ylabel={$x_2$},
			xlabel style={right},
			ylabel style={above},
			xtick={-4,4,8,12,16},
			ytick={-4,4,8,12},
			minor xtick={-4,-2,0,2,...,16}, %for grid
			minor ytick={-4,-2,...,14}, %for grid
			grid=minor,
			legend cell align={left}, 
			xmin=-4.9,
			xmax=17.9,
			ymin=-4.9,
			ymax=14.9]
				
	  \filldraw (2,2) circle (1.5pt) node[above] {$\bxh$}; % observed solution
		\draw[semithick,blue] (2,-4) -- (2,2) -- (8,-4) -- (2,-4); % imputed feasible region
		\draw[very thick,red,dashed] (2,-4) -- (4,-4) -- (2,-2) -- (2,-4); % prior feasible region
		\draw[->, semithick, blue] (2,-1) -- (3,-1) node[below] {$\bc$}; % cost vector
		
		% Legend
		\addlegendimage{very thick,red,dashed}
		\addlegendentry{Prior feasible region}
		\addlegendimage{semithick,blue}
		\addlegendentry{Imputed feasible region}
		
		\end{axis}
\end{tikzpicture}
\caption{Trivial solution averted by perturbing right-hand-side constraint vector.}
\label{fig:LpTrivialCon2}
\end{subfigure}
\begin{subfigure}{0.5\textwidth}
\centering
\begin{tikzpicture}
\begin{axis}[
			axis x line=center,
			axis y line=center,
			x=0.3cm, y=0.3cm, %sets the axis units, effectively controlling size of picture
			xtick={100},
			ytick={100},
			%minor xtick={-2,0,2,...,16}, %for grid
			%minor ytick={-4,-2,...,14}, %for grid
			%grid=minor,
			%legend cell align={left}, 
			xmin=-4.9,
			xmax=17.9,
			ymin=-4.9,
			ymax=14.9]
				
		\fill[fill=blue!10!white] (14,14) -- (2,2) -- (2,-4) -- (14,-4); % shading the imputed feasible region
		\end{axis}
		
		\begin{axis}[
			axis x line=center,
			axis y line=center,
			x=0.3cm, y=0.3cm, %sets the axis units, effectively controlling size of picture
			xlabel={$x_1$},
			ylabel={$x_2$},
			xlabel style={right},
			ylabel style={above},
			xtick={-4,4,8,12,16},
			ytick={-4,4,8,12},
			minor xtick={-4,-2,0,2,...,16}, %for grid
			minor ytick={-4,-2,...,14}, %for grid
			grid=minor,
			legend cell align={left}, 
			xmin=-4.9,
			xmax=17.9,
			ymin=-4.9,
			ymax=14.9]
				
	  \filldraw (2,2) circle (1.5pt) node[right] {$\bxh$}; % observed solution
		\draw[semithick,blue] (14,14) -- (2,2) -- (2,-4) -- (16,-4); % imputed feasible region
		\draw[very thick,red,dashed] (2,-4) -- (4,-4) -- (2,-2) -- (2,-4); % prior feasible region
		\draw[->, semithick, blue] (2,1) -- (4,1) node[right] {$\bc$}; % cost vector
		
		% Legend
		\addlegendimage{very thick,red,dashed}
		\addlegendentry{Prior feasible region}
		\addlegendimage{mark=square*,color=blue!10!white}
		\addlegendentry{Imputed feasible region}
		
		\end{axis}
\end{tikzpicture}
\caption{Trivial solution averted by perturbing prior left-hand-side constraint vector.}
\label{fig:LpTrivialCon3}
\end{subfigure}
\caption{A numerical example in which NLO-SD trivializes a constraint of the forward problem.}
\label{fig:LpTrivialCon}
\end{figure*}

We consider two methods to circumvent the trivial solution. First, if we add a small value of -0.1 to $b_3$, the third constraint is instead imputed as $-0.025x_1 -0.025x_2 \geq -0.1$, which is shown in Figure \ref{fig:LpTrivialCon2}. In this example, we perturb $b_3$ with a negative rather than positive value so that the imputed $\ba_3$ points in the same rather than opposite direction as the prior $\bah_3$. Second, if we add 0.1 to $\hat{a}_{31}$, the third constraint is instead adjusted to $0.05x_1 - 0.05x_2 \geq 0$, which again renders the feasible region unbounded (see Figure \ref{fig:LpTrivialCon3}). In this example, adjusting the weights $\bxi$ would not be useful, because those parameters only help determine which constraint of the forward problem is set active, whereas the third constraint in this case is violated by $\bxh$ and therefore needs to be adjusted regardless.
\end{exmp}

\end{document}